\documentclass[12pt]{amsart}

\usepackage[hmargin=0.8in,height=8.6in]{geometry}
\usepackage{amssymb,amsthm, times}
\usepackage{delarray,verbatim}
\usepackage[numbers]{natbib} 
\usepackage{ifpdf}
\ifpdf
\usepackage[pdftex]{graphicx}
 \else
\usepackage[dvips]{graphicx}
 \fi
\usepackage{epstopdf}
\usepackage{bm}

\usepackage{color}
\definecolor{webgreen}{rgb}{0,.5,0}
\definecolor{webbrown}{rgb}{.8,0,0}
\definecolor{emphcolor}{rgb}{0.95,0.95,0.95}

\usepackage{hyperref}
\hypersetup{%
	colorlinks=true,
	linkcolor=webbrown,
	filecolor=webbrown,
	citecolor=webgreen,
	breaklinks=true}
\ifpdf \hypersetup{pdftex,
	pdfstartview=FitH, 
	bookmarksopen=true,
	bookmarksnumbered=true
} \else \hypersetup{dvips} \fi

\newcommand {\ud}{{\rm d}}

\DeclareMathOperator{\hol}{C}

\numberwithin{equation}{section}
\theoremstyle{plain}
\newtheorem{theorem}{Theorem}[section]

\newtheorem{proposition}[theorem]{Proposition}
\newtheorem{lemma}[theorem]{Lemma}

\newtheorem{assump}[theorem]{Assumption}
\theoremstyle{remark}
\newtheorem{remark}[theorem]{Remark}

\newcommand {\R}{\mathbb{R}}

\newcommand {\E}{\mathbb{E}}

\newcommand{\diff}{{\rm d}}

\newcommand{\lev}{L\'{e}vy }

\title{Optimality of refraction strategies for a constrained dividend problem}
\author[M. Junca]{Mauricio Junca$^{(1)}$}
\thanks{$(1)$ Department of Mathematics, Universidad de los Andes, Bogot\'a, Colombia. Email: mj.junca20@uniandes.edu.co.}
\author[H. Moreno-Franco]{Harold Moreno-Franco$^{(2)}$}
\thanks{$(2)$ Laboratory of Stochastic Analysis and its Applications, National Research University Higher School of Economics, Moscow, Russia. Email:hmoreno@hse.ru.}
\author[J.L. P\'erez]{Jos\'e Luis P\'erez$^{(3)}$}
\thanks{$(3)$ Department of Probability and Statistics, Centro de Investigaci\'on en Matem\'aticas A.C.,Guanajuato, Mexico. Email: jluis.garmendia@cimat.mx.}
\author[K. Yamazaki]{Kazutoshi Yamazaki$^{(4)}$}
\thanks{$(4)$ Department of Mathematics, Faculty of Engineering Science, Kansai University, Osaka, Japan. Email: kyamazak@kansai-u.ac.jp.}
\begin{document}
\maketitle	

\begin{abstract}
We consider de Finetti's problem for spectrally one-sided L\'evy risk models with control strategies that are absolutely continuous with respect to the Lebesgue  measure. Furthermore, we consider the version with a constraint on the time of ruin.  To characterize the solution to the aforementioned models, we first solve the optimal dividend problem with a terminal value at ruin and show the optimality of threshold strategies. Next, we introduce the dual Lagrangian problem and show that the complementary slackness conditions are satisfied, characterizing the optimal Lagrange multiplier. Finally, we illustrate our findings with a series of numerical examples.\\
	\noindent \small{\noindent  AMS 2010 Subject Classifications: 60G51, 93E20, 91B30 
	}
	
\end{abstract}

\section{Introduction} 

In de Finetti's optimal dividend problem, the aim is to maximize the total expected discounted dividends accumulated until ruin. Intuitively, as the risk of ruin must be considered, dividends should be paid only when there is sufficient amount of surplus available. 
With this conjecture and under the assumption of stationary increments of the underlying process (in the \lev cases), the optimality of a barrier strategy that pays out any amount above a certain barrier has been pursued in various papers. Because the resulting controlled process becomes a classical reflected process, existing fluctuation theoretical results have been efficiently applied to solve explicitly the problem, at least under suitable conditions. See, among others, Avram et al.\ \cite{AvPaPi07} for the spectrally negative case and Bayraktar et al.\ \cite{BKY} for the spectrally positive case.
 
Despite these important works, there are several disputes about the classical model in the sense that the set of admissible strategies is too large and contains those that are in reality impossible to implement.  In particular, under the barrier strategy that is shown to be optimal, ruin must occur almost surely, and this is rather an undesirable conclusion. For these reasons, in the past decade, several extensions have been considered so as to obtain a more realistic model, by considering more restricted sets of admissible strategies and  modifications to the objective function so as to incorporate more directly the risk of ruin.
  
Motivated by these, in this paper we focus on the model with the absolutely continuous condition on the dividend strategy and additional condition on the time of ruin. We consider both cases driven by spectrally negative and positive \lev processes.
 
Regarding the absolutely continuous condition, it is assumed that the rate at which dividends are paid is bounded. More specifically, the dividend strategy must be absolutely continuous with respect to the Lebesgue measure with its density bounded by a given constant. Analogously to the barrier strategy that is optimal in the classical case, the \emph{threshold strategy} -- that pays dividends at the maximal rate as long as the surplus is above a certain fixed level  -- is optimal in this case. For a spectrally negative L\'evy surplus process, Kyprianou et al.\ \cite{KLP} showed the optimality of the threshold strategy under a completely monotone assumption on the \lev measure. The spectrally positive \lev model has been solved by Yin et al.\ \cite{Yin}. In both cases,  the optimally controlled process becomes the \emph{refracted \lev process} of Kyprianou and Loeffen \cite{KyLo}, and the fluctuation identities for this process can be used efficiently to solve de Finetti's problem under the absolutely continuous condition. 

Following the recent work of Hern\'andez et al. \cite{HJM17}, we study the case in which the longevity feature is added to the problem  by considering a constraint on the time of ruin. The longevity aspect of the firm remained as a separate problem; see \cite{schmidli2002} for a survey on this matter. Despite efforts to integrate both features \cite{Jostein03,ThonAlbr,Grandits}, it was not until very recently  a successful solution to a model that actually accounts for the trade-off between performance and longevity was presented. Hern\'andez and Junca \cite{HJ15} considered de Finetti's problem in the setting of Cram\'er-Lundberg reserves with i.i.d. exponentially distributed jumps adding a constraint to the expected time of ruin of the firm. 

The contribution of this paper is twofold:
\begin{enumerate}
	\item We first solve the optimal dividend problem \emph{with a terminal value at ruin} under the absolutely continuous assumption. We solve this problem for the spectrally negative \lev case under the assumption that the \lev measure has a completely monotone density and also for the general spectrally positive \lev case. In both models we show that a threshold strategy is optimal (see Theorems \ref{L.V.1} and \ref{L.V.2}). The optimal refraction level as well as the value function are concisely expressed in terms of the scale function. Its optimality is confirmed by a verification lemma.
	\item We then use these results to solve the constrained dividend maximization problem over the set of strategies such that the Laplace transform of the ruin time must be bounded by a given constant.  This is an extension of \cite{HJM17} under the absolutely continuous assumption. Theorem \ref{main.1} shows the result when the reserves are modeled by a spectrally negative L\'evy process with a completely monotone L\'evy density and Theorem \ref{main.2} for the general dual model. 
\end{enumerate}

The rest of the paper is organized as follows. In Section \ref{formProblem}, we formulate the problem. In Section \ref{section_scale_functions}, we present an overview of scale functions and some fluctuation identities related to spectrally negative \lev processes and their respective refracted processes. In Section  \ref{optimal_dividend_tv_sn}, we solve the optimal dividend problem with a terminal cost and the absolutely continuous assumption for the case of a spectrally negative L\'evy process with a completely monotone L\'evy density.  In Section \ref{const.1}, we extend the results to solve the constrained dividends problem. In Section \ref {specpos} we solve the same problems for the spectrally positive case. Finally, in Section \ref{numerical_section}, we give some numerical results.
\section{Formulation of the problem}\label{formProblem}

In this section, we formulate the constrained de Finetti's problem driven by a spectrally negative \lev process. The spectrally positive \lev process is its dual and a slight modification is only needed to formulate the spectrally positive case (see Section \ref{specpos}).

\subsection{Spectrally negative \lev processes }
Recall that a spectrally negative L\'evy process is a stochastic process, which has c\`adl\`ag paths and stationary and independent increments such that there are no positive discontinuities. To avoid degenerate cases in the forthcoming discussion, we shall additionally exclude from this definition the case of monotone paths. This means that we are not interested in the case of a deterministic increasing linear drift or the negative of a subordinator. Henceforth we assume that $X=\{X_t: t\ge 0\}$ is a spectrally negative L\'evy process defined on a probability space $(\Omega,\mathcal{F},\mathbb{P})$ with L\'evy triplet given by $(\gamma, \sigma, \Pi)$, where $\gamma \in \R$, $\sigma\ge 0$ and $\Pi$ is a measure concentrated on $(0,\infty)$ satisfying
\begin{equation}\label{L.1}
\int_{(0,\infty)}(1\wedge z^2)\Pi({\rm d}z)<\infty.
\end{equation}

The Laplace exponent of $X$ is given by
\begin{equation} \label{laplace_exponent}
\psi(\lambda)= \log \mathbb{E} \left[ \mathrm{e}^{\lambda X_1} \right] = \gamma \lambda +\frac12\sigma^2\lambda^2 -
\int_{(0,\infty)} \left( 1-\mathrm{e}^{-\lambda z} -\lambda z\mathbf{1}_{\{0<z\leq1\}} \right)\Pi(\mathrm{d}z),
\end{equation}
which is well defined for $\lambda\geq0$. Here $\mathbb{E}$ denotes the expectation with respect to $\mathbb{P}$. The reader is noted that, for convenience, we have arranged the representation of the Laplace exponent in such a way that the support of the L\'evy measure is positive even though the process experiences only negative jumps. As a strong Markov process, we shall endow $X$ with probabilities $\{\mathbb{P}_x : x\in\mathbb{R}\}$ such that under $\mathbb{P}_x$ we have $X_0 = x$ with probability one. Note that $\mathbb{P}_0= \mathbb{P}$ and $\mathbb{E}_0= \mathbb{E}$.

It is well-known that $X$ has paths of bounded variation if and only if $\sigma=0$ and $\displaystyle\int_{(0,1]} z\Pi(\mathrm{d}z)<\infty$. In this case $X$ can be written as
\begin{equation*}
X_t=ct-S_t, \,\, t\geq 0,
\end{equation*}
where $c:=\gamma+\displaystyle\int_{(0,1]}z\Pi(\mathrm{d}z)$ and $\{S_t:t\geq0\}$ is a drift-less subordinator. Note that we must have 
$c>0$, since 
it is assumed that $X$ does not have
monotone paths. 

The process $X$ is a generalization of what is known in the actuarial mathematics literature as the classical Cram\'er-Lundberg risk process. This process is often used to  model the surplus wealth of an insurance company.

\subsection{Admissible strategies}
Let $D=\{ D_t: t \geq 0 \}$ be a dividend strategy, meaning that it is a  non-negative and non-decreasing process adapted to the completed and right continuous filtration $\mathbb{F}:=\{\mathcal{F}_t : t\geq 0\}$ of $X$.  Here, for each fixed $t\geq0$, the quantity $D_t$ represents the cumulative dividends paid out up to time $t$ by the insurance company whose risk process is modeled by $X$. The controlled L\'evy process becomes $U^D=\{U^D_t= X_t - D_t :t\geq 0\}$ and we write
\begin{align*}
\tau^D := \inf\{t>0: U^D_t < 0 \},
\end{align*}
 for the time at which ruin occurs when the dividend payments are taken into account.  

In this work we are interested in adding a constraint to the dividend processes. Specifically, we will only work with absolutely continuous strategies of bounded rate, i.e., 
\begin{equation*}
D_t=\int_0^{t}d(s)\mathrm{d}s,\ t\geq0,
\end{equation*}
such that  the \textit{dividend rate}  $d$ satisfies $0\leq d(t)\leq \delta$, for $t\geq0$, where $\delta>0$ is a ceiling rate. We will denote by $\Theta$ the family of admissible strategies satisfying the conditions mentioned above.

\subsection{Constrained de Finetti's problem and  its dual} \label{subsection_constrained_de_finetti}
The expected net present value under the dividend policy $D\in\Theta$ with discounting at rate $q>0$ and initial capital $x\geq 0$ is given by 
\begin{equation*}
v^{D}(x) = \mathbb{E}_x\left[\int_0^{\tau^{D}} \mathrm{e}^{-qt} \mathrm{d}D_t\right].
\end{equation*}
The dividend problem, originally considered by de Finetti, asks to maximize the expected net present value of dividend payments over the set of strategies $\Theta$. 

Now, as studied in \cite{HJM17}, we are interested in addressing a modification of this problem by adding a restriction to the dividend process $D$, which is given by the following constraint
\begin{equation*}
\mathbb{E}_x\left[e^{-q\tau^D}\right]\leq K,\qquad 0 \leq  K\leq 1\text{ fixed.}
\end{equation*}
Strategies in $\Theta$ satisfying this constraint are called \textit{feasible}, and are called \textit{infeasible} otherwise.

We want to maximize the expected net present value of dividend payments over the set of feasible strategies. That is, we aim to solve the optimization problem, for $x\geq0$ and $0 \leq K \leq 1$,
\begin{equation}\label{F.1}
V(x;K):=\sup_{D\in\Theta}v^D(x),\quad\text{s.t.}\quad \mathbb{E}_x\left[e^{-q\tau^D}\right]\leq K,
\end{equation}
where,  in the case $\mathbb{E}_x [e^{-q\tau^D} ] > K$ for all $D \in \Theta$, we set $V(x;K)= -\infty$ and call the problem  \eqref{F.1} infeasible.

Proceeding as in \cite{HJM17}, we use Lagrange multipliers to reformulate the problem. For $\Lambda\geq0$ we define the function
\begin{equation}\label{lm}
v_{\Lambda}^D(x;K):=v^D(x)+\Lambda\left(K-\mathbb{E}_x\left[e^{-q\tau^D}\right]\right).
\end{equation}
Note that we can write the problem \eqref{F.1} as $V(x;K)=\displaystyle\sup_{D\in\Theta}\displaystyle\inf_{\Lambda\geq0}v_{\Lambda}^D(x;K)$ since any infeasible strategy $D$ will make 
$\displaystyle\inf_{\Lambda\geq0}v_{\Lambda}^D(x;K)=-\infty$, and any feasible strategy $D$ will make
$\displaystyle\inf_{\Lambda\geq0}v_{\Lambda}^D(x;K)=v^D_0 (x;K)=v^D(x)$.
\par The dual problem of \eqref{F.1} is obtained by interchanging the $\sup$ with the $\inf$ in the expression above, yielding an upper bound,
\begin{equation}\label{p_dual}
V(x;K)=\sup_{D\in\Theta}\inf_{\Lambda\geq0}v_{\Lambda}^D(x;K)\leq \inf_{\Lambda\geq0} V_{\Lambda}(x;K),
\end{equation}
where 
\begin{align}
V_{\Lambda}(x;K) := \sup_{D\in\Theta} v_{\Lambda}^D(x;K). \label{p_Lambda}
\end{align}
Therefore, the main goal is to prove that $V(x;K)= \inf_{\Lambda\geq0}\displaystyle V_{\Lambda}(x;K)$ and to find an optimal $\Lambda$ (Lagrange multiplier) with which the infimum is attained. In order to do this we will first solve \eqref{p_Lambda}. 
\par We remark that if we set 
\begin{align}
V_\Lambda(x) :=V_\Lambda(x; 0) \quad \textrm{and} \quad v_\Lambda^D(x) := v_\Lambda^D(x; 0), \quad D \in \Theta, \label{V_Lambda_def}
\end{align}
then $v_\Lambda^D(x; K) = v_\Lambda^D(x) + \Lambda K$ and hence $V_\Lambda(x; K) =V_\Lambda(x) + \Lambda K$. Therefore, solving \eqref{p_Lambda} is equivalent to solving
\begin{equation}\label{p_Lambda_zero}
V_{\Lambda}(x):=\sup_{D\in\Theta}v_{\Lambda}^D(x).
\end{equation}

\section{Review of scale functions}  \label{section_scale_functions}
In this section we review the scale function of spectrally negative \lev processes. First, we define the process $Y=\{Y_t=X_t-\delta t: t\ge 0\}$ with its Laplace exponent 
\begin{align}\label{lap_exp_Y}
\psi_Y(\theta) := \psi(\theta) - \delta \theta, \quad \theta \geq 0. 
\end{align}
We assume here that $Y$ is a spectral negative \lev process and not the negative of a subordinator (see Assumption \ref{assumpdelta}).

Fix $q > 0$. Following the same notations as in \cite{KyLo}, we use $W^{(q)}$ and $\mathbb{W}^{(q)}$ for the scale functions of $X$ and $Y$ respectively.  These are the mappings from $\R$ to $[0, \infty)$ that are zero on the negative half-line, while on the positive half-line they are strictly increasing functions that are defined by their Laplace transforms:
\begin{align} \label{scale_function_laplace}
\begin{split}
\int_0^\infty  \mathrm{e}^{-\theta x} W^{(q)}(x) \diff x &= \frac 1 {\psi(\theta)-q}, \quad \theta > \Phi(q), \\
\int_0^\infty  \mathrm{e}^{-\theta x} \mathbb{W}^{(q)}(x) \diff x &= \frac 1 {\psi_Y(\theta) -q}, \quad \theta > \varphi(q),
\end{split}
\end{align}
where 
\begin{align}\label{def_varphi}
\Phi(q) := \sup \{ \lambda \geq 0: \psi(\lambda) = q\} \quad \textrm{and} \quad \varphi(q) := \sup \{ \lambda \geq 0: \psi_Y(\lambda) = q\}.
\end{align}
By the strict  convexity of $\psi$, we derive the inequality $\varphi(q) > \Phi(q) > 0$.
\par We also define, for $x \in \R$,
\begin{align*}
\overline{W}^{(q)}(x) &:=  \int_0^x W^{(q)}(y) \diff y, \\
Z^{(q)}(x) &:= 1 + q \overline{W}^{(q)}(x),  \\
\overline{Z}^{(q)}(x) &:= \int_0^x Z^{(q)} (z) \diff z = x + q \int_0^x \int_0^z W^{(q)} (w) \diff w \diff z.
\end{align*}
Noting that $W^{(q)}(x) = 0$ for $-\infty < x < 0$, we have
\begin{align*}
\overline{W}^{(q)}(x) = 0, \quad Z^{(q)}(x) = 1  \quad \textrm{and} \quad \overline{Z}^{(q)}(x) = x, \quad x \leq 0.  
\end{align*}
In addition, we define $\overline{\mathbb{W}}^{(q)}$, $\mathbb{Z}^{(q)}$ and $\overline{\mathbb{Z}}^{(q)}$ analogously for $Y$.  The scale functions of $X$ and $Y$ are related, for $x \in \R$,  by the following equality
\begin{align}\label{RLqp}
&\delta\int_0^x\mathbb{W}^{(q)}(x-y)W^{(q)}(y)\ud y=\overline{\mathbb{W}}^{(q)}(x)-\overline{W}^{(q)}(x), 
\end{align} 
which can be proven by showing that the Laplace transforms on both sides are equal. 

Regarding their behaviors as $x \downarrow 0$, we have, as in Lemmas 3.1 and 3.2 of      
\cite{KKR},
\begin{align}\label{eq:Wqp0}
\begin{split}
W^{(q)} (0) &= \left\{ \begin{array}{ll} 0, & \textrm{if $X$ is of unbounded
	variation,} \\ c^{-1}, & \textrm{if $X$ is of bounded variation,}
\end{array} \right.  \\
\mathbb{W}^{(q)} (0) &= \left\{ \begin{array}{ll} 0, & \textrm{if $Y$ is of unbounded
	variation,} \\ (c-\delta)^{-1}, & \textrm{if $Y$ is of bounded variation,}
\end{array} \right. 
\end{split}
\end{align} 
and 
\begin{align} \label{W_zero_derivative}
\begin{split}
W^{(q)\prime} (0+) &:=\lim_{x \downarrow 0}W^{(q) \prime} (x) =
\left\{ \begin{array}{ll}  \dfrac{2}{\sigma^2}, & \textrm{if }\sigma > 0, \\
\infty, & \textrm{if }\sigma = 0 \; \textrm{and} \; \Pi(0, \infty) = \infty, \\
\dfrac{q + \Pi(0, \infty)} {c^2}, &  \textrm{if }\sigma = 0 \; \textrm{and} \; \Pi(0, \infty) < \infty,
\end{array} \right. \\
\mathbb{W}^{(q) \prime }(0+) &:= \lim_{x \downarrow 0}\mathbb{W}^{(q)\prime} (x) =
\left\{ \begin{array}{ll}  \dfrac {2} {\sigma^2}, & \textrm{if }\sigma > 0, \\
\infty, & \textrm{if }\sigma = 0 \; \textrm{and} \; \Pi(0,\infty) = \infty, \\
\dfrac {q + \Pi(0, \infty)} {(c-\delta)^2}, &  \textrm{if }\sigma = 0 \; \textrm{and} \; \Pi(0,\infty) < \infty.
\end{array} \right. 
\end{split}
\end{align}
On the other hand, as in Lemma 3.3 of \cite{KyprianouRS10}, 
\begin{align}\label{W_q_limit}
e^{-\Phi(q) x}W^{(q)} (x) \nearrow \psi'(\Phi(q))^{-1} \quad \textrm{and} \quad e^{-\varphi(q) x}\mathbb{W}^{(q)} (x) \nearrow \psi_Y'(\varphi(q))^{-1}, \quad \textrm{as } x \rightarrow \infty.
\end{align}

\section{Optimal dividend problem with a terminal value}\label{optimal_dividend_tv_sn}

In this section, we solve the problem \eqref{p_Lambda_zero}. The results obtained here are applied to the constrained case in the next section. In this and next sections where we deal with the spectrally negative case, we assume the following.

\begin{assump} \label{assump_completely_monotone}
The \lev measure $\Pi$ of the process $X$ has a completely monotone density.  That is, $\Pi$ has a density $\pi$ whose $n^{th}$ derivative $ \pi^{(n)}$ exists for all $n \geq 1$ and satisfies
\begin{align*}
(-1)^n \pi^{(n)} (x) \geq 0, \quad x > 0.
\end{align*}
\end{assump}

\begin{remark} \label{remark_smoothness_scale_function_SN}
Under Assumption \ref{assump_completely_monotone}, the scale functions $W^{(q)}$ and $\mathbb{W}^{(q)}$ defined in Section \ref{section_scale_functions} are infinitely continuously differentiable on $(0, \infty)$. For more details, see Lemma \ref{compmon}.
\end{remark}

This assumption is known to be a sufficient optimality condition for threshold strategies in the classical spectrally negative case by Loeffen \cite{Loeffen09}, for the absolutely continuous case (with $\Lambda=0$) by Kyprianou et al. \cite{KLP}, and for the periodic case by Noba et al. \cite{NPYY} (with $\Lambda=0$).

In this section, \emph{we allow $\Lambda$ to be negative} (in which case a positive payoff is collected at ruin time) but need to assume the following in order to avoid the trivial case (see Remark \ref{remark_trivial}).
\begin{assump} \label{assump_SN_param}
We assume $q \Lambda + \delta > 0$.
\end{assump}
\begin{remark} \label{remark_trivial}
In the case Assumption \ref{assump_SN_param} does not hold (i.e.\ $q \Lambda + \delta \leq 0$), because the dividend rate is bounded by $\delta$, we have that, for any dividend policy $D\in\Theta$ and $x\geq 0$,
\begin{align*}
v^D_{\Lambda}(x)\leq \delta \E_x \left[\int_0^{\tau^D} e^{-qt } \diff t\right] + - \Lambda \E_x [e^{-q \tau^D}] &\leq \frac{\delta}{q}-\frac{q\Lambda+\delta}{q}\mathbb{E}_x\left[e^{-q\tau^D}\right]
&\leq \frac{\delta}{q}-\frac{q\Lambda+\delta}{q}
\sup_{D' \in \Theta}\mathbb{E}_x\left[e^{-q\tau^{D'}}\right].
\end{align*}

This implies that $v^D_{\Lambda}$ is maximized by taking the strategy that pays dividends at the ceiling rate $\delta$ for all $t\geq0$ because it maximizes $\mathbb{E}_x\left[e^{-q\tau^{D'}}\right]$ over $D' \in \Theta$.
\end{remark}

Finally, we make the following assumption:
\begin{assump}\label{assumpdelta}
If $X$ has paths of bounded variation, then $\delta<c$.
\end{assump}

This assumption is commonly assumed in the literature (see \cite{KLP} and \cite{PYX}). This is needed so that one cannot completely reflect the process at a given barrier -- otherwise, the problem is almost identical to the classical case without the absolutely continuous assumption.

\subsection{Threshold strategies}
The objective of this section is to show that the optimal strategies for \eqref{p_Lambda_zero} are of the threshold type. Under the threshold strategy $D^b$ for $b \geq 0$, the resulting controlled process $U^b$ is known as a refracted L\'evy process of  \cite{KyLo}, which is the unique strong solution to 
\[
U^b_t= X_t - D_t^b \quad \textrm{ where} \quad 
D^b_t := \delta\int_0^t1_{\{U_s^b>b\}}\mathrm{d}s,\qquad\text{$t\geq0$.}
\]
Let its ruin time be
\begin{align*}
\tau_{b}:=\inf\{t>0: U^{b}_t < 0 \}.
\end{align*}

The next identities are lifted from Theorems 5.(ii) and 6.(ii) in \cite{KyLo}. For $x \in \R$ and $b \geq 0$, we have
\begin{align}\label{div_ac}
\mathbb{E}_x\left[\int_0^{\tau_{b}} e^{-qt}
\diff D^b_t\right]=-\delta \overline{\mathbb{W}}^{(q)}(x-b)+ \frac{1}{h(b)}\biggl(W^{(q)}(x)+\delta\int_b^x \mathbb W^{(q)}(x-y)W^{(q)\prime}(y)\diff y\biggr),
\end{align}
and
\begin{align}\label{lap_ac}
\Psi_{x}(b)&:=  \mathbb{E}_{x}\left[e^{-q\tau_{b}}\right]\notag\\
&=Z^{(q)}(x)+\delta q \int_b^x\mathbb{W}^{(q)}(x-y)W^{(q)}(y)\diff y\notag\\
&\quad -\frac{q\varphi(q)e^{\varphi(q)b}\displaystyle\int_b^{\infty}e^{-\varphi(q) y}W^{(q)}(y) \diff y}{h(b)} \left(W^{(q)}(x)+\delta\int_b^x\mathbb{W}^{(q)}(x-y)W^{(q)\prime}(y)\diff y\right),
\end{align}
where
\begin{equation}\label{funh}
h(b):=\varphi(q)e^{\varphi(q)b}\displaystyle\int_{b}^{\infty}e^{-\varphi(q)y}W^{(q)\prime}(y)\diff y.
\end{equation}
Under the threshold strategy $D^b$, the expected net present value is denoted by
\begin{equation}\label{vf_1_new} 
v_{\Lambda}^{b}(x):=\mathbb{E}_x\left[\int_0^{\tau_{b}} e^{-qt}
\diff D^b_t\right]-\Lambda\Psi_{x}(b),\ \text{for}\ x\geq0. 
\end{equation}

Using \eqref{div_ac} and \eqref{lap_ac}, we have the following result. 
\begin{proposition}
The function $v_{\Lambda}^{b}$, with  $b\geq 0$, is given by
\begin{align}\label{vf_1}
v_{\Lambda}^{b}(x)&=\xi_{\Lambda}(b)\left(W^{(q)}(x)+\delta\int_b^x\mathbb{W}^{(q)}(x-y)W^{(q)\prime}(y)\diff y\right)\notag\\
&-\Lambda\left(Z^{(q)}(x)+\delta q \int_b^x\mathbb{W}^{(q)}(x-y)W^{(q)}(y)\diff y\right)-\delta\overline{\mathbb W}^{(q)}(x-b),\ \text{for $x\geq 0$,}
\end{align}
where 
\begin{align}
\xi_{\Lambda}(b):=\frac{1}{h(b)}\biggl(1+ \varphi(q) q\Lambda e^{\varphi(q)b}\displaystyle\int_{b}^{\infty}e^{-\varphi(q) y}W^{(q)}(y) \diff y\biggr)\label{xi_lambda}.
\end{align}
In particular, for $x \leq b$, we have
\begin{align}\label{vf_1_x_less}
v_{\Lambda}^{b}(x)&=\xi_{\Lambda}(b) W^{(q)}(x) -\Lambda Z^{(q)}(x).
\end{align}
\end{proposition}
\begin{remark}
From \eqref{funh} and integration by parts,
\begin{align}
\varphi(q)e^{\varphi(q)b}\int_b^\infty e^{-\varphi(q) y} W^{(q)} (y) \diff  y = W^{(q)} (b) + \frac{h(b)}{\varphi(q)}, \quad b \geq 0. \label{int_by_parts_laplace}
\end{align}
Hence, the function $\xi_{\Lambda}$,  given in \eqref{xi_lambda}, can be rewritten in the following way
\begin{equation}\label{v5}
\xi_{\Lambda}(b)=\frac{1}{h(b)}\biggr(1+q\Lambda\bigg(W^{(q)}(b)+\frac{h(b)}{\varphi(q)}\bigg)\biggl), \quad b \geq 0.
\end{equation}
\end{remark}
In particular, for the case $b = 0$, these expressions can be simplified as follows; the proof is deferred to Appendix \ref{proof_lemma_zero_case}.
\begin{lemma} \label{lemma_zero_case}
We have 
\begin{equation}\label{xi_0_new}
\begin{split}
h(0)&=\varphi(q) \displaystyle\int_{0}^{\infty}e^{-\varphi(q)y}W^{(q)\prime}(y)\diff y = \varphi(q) \Big( - W^{(q)}(0)  + \delta^{-1} \Big), \\
\xi_{\Lambda}(0)&=\frac{1}{h(0)}\biggl(1+ \frac {q \Lambda} \delta \bigg)= \dfrac{\delta+q\Lambda}{\varphi(q)(1-\delta W^{(q)}(0))},
\end{split}
\end{equation}
and, for $x\geq0$,
\begin{align}\label{v_0_new}
v_{\Lambda}^{0}(x)&=\xi_{\Lambda}(0) \mathbb{W}^{(q)}(x)  (1 -  \delta W^{(q)}(0)) -\Lambda \mathbb{Z}^{(q)}(x) -\delta \overline{\mathbb W}^{(q)}(x).
\end{align}
\end{lemma}

\subsection{Selection of optimal threshold $b_\Lambda$}
 Focusing on the set of threshold strategies, we now select our candidate optimal threshold, which we call $b_\Lambda$.  In view of \eqref{vf_1_x_less}, such $b_\Lambda$ must maximize $\xi_\Lambda$. Motivated by this fact, we pursue $b_\Lambda$ such that $\xi'_\Lambda(b_{\Lambda})$ vanishes if such a value exists.

First, we rewrite the form of $\xi_{\Lambda}'(b)$ as follows. Fix $b > 0$. Taking a derivative in \eqref{v5} and using that $h'(b)=\varphi(q) (h(b) - W^{(q)\prime}(b))$ (by \eqref{funh}), 
\begin{align}\label{v9}
\xi_{\Lambda}'(b) & = q\Lambda-\dfrac{h'(b)}{h(b)}\xi_{\Lambda}(b)  \notag\\
& =q\Lambda-\frac{\varphi(q)} {h(b)}\biggr(1+q\Lambda\bigg(W^{(q)}(b)+\dfrac{h(b)}{\varphi(q)}\bigg)\biggl)+\dfrac{\varphi(q) W^{(q)\prime}(b)}{h(b)}\xi_{\Lambda}(b) \notag\\
& =\dfrac{\varphi(q)W^{(q)\prime}(b)}{h(b)}(\xi_{\Lambda}(b)-g_{\Lambda}(b)), 
\end{align}
where 
\begin{equation}\label{v9.0}
g_{\Lambda}(b):=\dfrac{1+q\Lambda W^{(q)}(b)}{W^{(q)\prime}(b)}. 
\end{equation}

In view of \eqref{v9}, we now define the (candidate) optimal barrier level for \eqref{p_Lambda_zero} by
\begin{equation}\label{b1}
b_\Lambda := \inf \{ b > 0: \xi_\Lambda'(b) \leq 0 \} = \inf \{ b > 0: \xi_\Lambda(b) - g(b) \leq 0 \}.
\end{equation}
Here, we set $\inf \emptyset = \infty$ for convenience, but we will see in Proposition \ref{maxb.1} that $b_\Lambda$ is necessarily finite.

\begin{remark}\label{remfunh}
Following the proof of Lemma 3 in \cite{KLP}, the function $h$ as in \eqref{funh} has the following properties:
\begin{itemize}
\item[(i)] As a special case with $\Lambda = 0$,  $b_{0}$ is the point where $b\mapsto h(b)$ attains its global minimum. Hence $h'(b)<0$ for $b<b_0$ and $h'(b)>0$ for $b>b_0$.
\item[(ii)] We have that $\displaystyle\lim_{b\to\infty}h(b)=\infty$.
\end{itemize}
\end{remark}
\begin{remark} \label{remark_about_g}
Note that the function $g_\Lambda$ plays a key role in \cite{Loeffen08} and satisfies the following:
\begin{enumerate}
\item $g_{\Lambda}(0+)=\dfrac{1+q\Lambda W^{(q)}(0)}{W^{(q)\prime}(0 +)}$ and $g_{\Lambda}(b)\rightarrow\dfrac{q
		\Lambda}{\Phi(q)}$ as $b\rightarrow\infty$ by \eqref{W_q_limit}.
\item If we define 
\begin{align*}
a_{\Lambda} := \sup\{b\geq0: g_{\Lambda}(b)\geq g_{\Lambda}(x), \ \text{for all}\ x\geq0\},
\end{align*}
we know that $a_{\Lambda}$ is finite (see \cite[Proposition 3]{Loeffen08}) and is the unique point where $g_{\Lambda}$ has a global maximum; see \cite[proof of Thm.1]{Loeffen08}. Moreover if $a_{\Lambda}=0$, then $g'_{\Lambda} (b)<0$ for $b \in (0,\infty)$, and if $a_{\Lambda}>0$ then $g'_{\Lambda}(a_{\Lambda})=0$, $g'_{\Lambda}(b)>0$ for $b<a_{\Lambda}$ and $g'_{\Lambda}(b)<0$ for $b>a_{\Lambda}$.
\end{enumerate}
\end{remark}

We will now prove an auxiliary result which describes the asymptotic behaviour of the function $\xi_{\Lambda}$.
\begin{lemma} We have
\begin{equation}\label{lim}
\lim_{b\to\infty}\xi_{\Lambda}(b)=\lim_{b\to\infty}g_{\Lambda}(b)=\dfrac{q\Lambda}{\Phi(q)}.
\end{equation}	
\end{lemma}
\begin{proof}
Recall from Remark \ref{remark_about_g} the convergence of $g_\Lambda$. Now, letting $b\rightarrow\infty$ in \eqref{v5}, we observe that 
\begin{equation}\label{v7}
\lim_{b\rightarrow\infty}\xi_{\Lambda}(b)=q\Lambda\biggl(\frac{1}{\varphi(q)}+\lim_{b\rightarrow\infty}\frac{ W^{(q)}(b)}{h(b)}\biggr),
\end{equation}
since $h(b)\rightarrow\infty$ as $b\rightarrow\infty$. On the other hand, by dominated convergence theorem and using \eqref{W_q_limit}, it follows that  
\begin{align}\label{v8}
\dfrac{W^{(q)}(b)}{h(b)} &= \biggr(\varphi(q) \int_{0}^{\infty}e^{-(\varphi(q)-\Phi(q)) y} \frac {W^{(q)\prime} (y+b)} {W^{(q)}(y+b)}\dfrac{e^{-\Phi(q)[y+b]}W^{(q)}(y+b)}{e^{-\Phi(q)b}W^{(q)}(b)}\diff y \biggl)^{-1}\notag\\ 
&\xrightarrow{b \uparrow \infty} \biggl( \varphi(q)\Phi(q) \int_0^\infty e^{-(\varphi(q)-\Phi(q)) y} \diff y  \biggr)^{-1}  = \dfrac{\varphi(q)-\Phi(q)}{\varphi(q)\Phi(q)},
\end{align}
where we recall that $\varphi(q)>\Phi(q)$. Now, applying \eqref{v8} in \eqref{v7}, we get \eqref{lim}.
\end{proof}
\begin{proposition}\label{maxb.1}
(1) Under Assumption \ref{assump_completely_monotone}, we have $b_{\Lambda}\in[0,a_\Lambda]$.

(2) Moreover, $b_{\Lambda}=0$ if and only if one of the following two cases holds: 
\begin{enumerate}
	\item[\textit{(i)}] $\sigma=0$, $\Pi(0,\infty)<\infty$, and $\varphi(q)\geq \dfrac{(\delta+q\Lambda)(q+\Pi(0,\infty))}{(c+q\Lambda)(c-\delta)} =: \phi_1(\Lambda)$, or
	\item[\textit{(ii)}]$\sigma>0$ and $\varphi(q)\geq\dfrac{2(\delta+q\Lambda)}{\sigma^{2}} =: \phi_2(\Lambda)$.  
\end{enumerate}
\end{proposition}
\begin{proof}
(1) By the definition of $b_\Lambda$ as in \eqref{b1} and the continuity of $\xi_\Lambda$ and $g_\Lambda$,  in order to show $b_\Lambda \leq a_\Lambda$ it is sufficient to show $\xi_{\Lambda}'(b) \leq 0$ (or equivalently $\xi_\Lambda(b) - g_\Lambda(b) \leq 0$) on $[a_\Lambda, \infty)$.  To show this, suppose there exists $\bar{b} > a_\Lambda$ such that $\xi_\Lambda(\bar{b}) - g_\Lambda (\bar{b})> 0$.  Then, since $g_{\Lambda}$ is decreasing on $(a_\Lambda, \infty)$ as in Remark \ref{remark_about_g} (2),  we obtain by \eqref{v9} that $\xi_\Lambda(b) - g_\Lambda(b)$ is increasing on $(\bar{b}, \infty)$. However, this contradicts with \eqref{lim}. Hence,  $\xi_{\Lambda}'(b) \leq 0$ on $[a_\Lambda, \infty)$.

(2) Using \eqref{v9} and the definition of $b_{\Lambda}$ given in \eqref{b1}, we obtain that  $b_{\Lambda}=0$ if and only if $g_{\Lambda}(0+) \geq \xi_{\Lambda}(0+)$. This is equivalent, by Lemma \ref{lemma_zero_case} and Remark \ref{remark_about_g} (1), to 
	$$\varphi(q) \geq \dfrac{(\delta+q\Lambda)W^{(q)\prime}(0+)}{(1+q\Lambda W^{(q)}(0))(1-\delta W^{(q)}(0))}.$$ 
From here, using \eqref{eq:Wqp0} and \eqref{W_zero_derivative}, we obtain the two cases announced in the proposition.
\end{proof}

\begin{remark} \label{remark_b_zero_all}
For the case $\sigma=0$ and $\Pi(0,\infty)<\infty$ and the case  $\sigma > 0$,  the functions $\phi_1$ and $\phi_2$, respectively, are both strictly increasing  (since $c>\delta$ in the bounded variation case by Assumption \ref{assumpdelta}), with $\phi_1(-\delta/q) = \phi_2(-\delta/q) = 0$. Hence there exists $\bar{\Lambda} \in (-\delta/q, \infty]$ such that
$b_{\Lambda}=0$ for all $- \delta/q < \Lambda \leq \bar{\Lambda}$ and $b_{\Lambda}>0$ for all $\Lambda>\bar{\Lambda}$.
\begin{enumerate}  
\item  Suppose $\sigma=0$ and $\Pi(0,\infty)<\infty$. Define 
\begin{align*}
\phi_{1}(\infty) := \lim_{\Lambda\to\infty}\phi_{1}(\Lambda)=\frac{q+\Pi(0,\infty)}{c-\delta}.
\end{align*}
  If $\varphi(q)\geq \phi_1(\infty)$, then by Proposition \ref{maxb.1} (2), $\bar{\Lambda} = \infty$.  Otherwise, we must have $\bar{\Lambda} < \infty$ and $\phi_1(\bar{\Lambda}) = \varphi(q)$.
\item  Suppose $\sigma > 0$. Because  $\displaystyle\lim_{\Lambda\to\infty}\phi_{2}(\Lambda)= \infty$, we must have  $\bar{\Lambda} < \infty$. This also implies  $\phi_2(\bar{\Lambda}) = \varphi(q)$.
\item  Suppose $\sigma = 0$ and $\Pi(0,\infty)=\infty$.  Then, we can set $\bar{\Lambda} = - \delta /q$. 
 \end{enumerate}

\end{remark}

\subsection{Verification}
For the case $b_\Lambda > 0$,  by how $b_\Lambda$ is selected as in \eqref{b1}, together with \eqref{vf_1} and \eqref{v9},  we can write
\begin{align} \label{value_func_simplified}
\begin{split}
v_{\Lambda}^{b_\Lambda}(x) &=g_{\Lambda}(b_\Lambda)\left(W^{(q)}(x)+\delta\int_{b_\Lambda}^x\mathbb{W}^{(q)}(x-y)W^{(q)\prime}(y)\diff y\right) \\
&-\Lambda\left(Z^{(q)}(x)+\delta q \int_{b_\Lambda}^x\mathbb{W}^{(q)}(x-y)W^{(q)}(y)\diff y\right)-\delta\overline{\mathbb W}^{(q)}(x-b_\Lambda),\ \text{for $x\geq 0$.}
\end{split}
\end{align}
For the case $b_\Lambda = 0$, the function $v_{\Lambda}^{b_\Lambda} \equiv v_{\Lambda}^{0}$ is given in \eqref{v_0_new}.

Given the spectrally negative L\'evy process $X$, we call a function $F: \R \to \R$ \emph{sufficiently smooth}, if $F$ is continuously differentiable on $(0,\infty)$ when $X$ has paths of bounded variation and is twice continuously differentiable on $(0,\infty)$ when $X$ has paths of unbounded variation. We let $\Gamma$ be the operator acting on a sufficiently smooth function $F$, defined by
\begin{equation*}
\begin{split}
\Gamma F(x) := \gamma F'(x)+\frac{\sigma^2}{2}F''(x) +\int_{(0,\infty)}(F(x-z)-F(x)+F'(x)z\mathbf{1}_{\{0<z\leq1\}})\Pi(\mathrm{d}z), \quad x > 0.
\end{split}
\end{equation*}
The following lemma constitutes standard technology as far as optimal control is concerned. For its proof we refer to that of Lemma 1 in \cite{Loeffen08}.

\begin{lemma}\label{verificationlemma}
Suppose $\hat{D}\in\Theta$ is an admissible dividend strategy such that $v^{\hat{D}}_{\Lambda}$ is sufficiently smooth on $(0,\infty)$, $v_\Lambda^{\hat{D}}(0)\geq -\Lambda$, and for all $x>0$,
\begin{equation}\label{HJB-inequality}
(\Gamma -q) v^{\hat{D}}_{\Lambda}(x)+\sup_{0\leq r\leq\delta}(r-rv^{\hat{D}\prime}_{\Lambda}(x))\leq 0.
\end{equation}
Then $v^{\hat{D}}_{\Lambda}(x)=V_{\Lambda}(x)$ for all $x\geq0$ and hence $\hat{D}$ is an optimal strategy.
\end{lemma}

We first show that our candidate value function $v^{b_{\Lambda}}_{\Lambda}$   is indeed sufficiently smooth on $(0,\infty)$. 
\begin{lemma} \label{smooth_fit_prob1}
Consider $b_{\Lambda}\geq0$ given by \eqref{b1}. Then $v^{b_{\Lambda}}_{\Lambda}$ is sufficiently smooth on $(0,\infty)$.
\end{lemma}
\begin{proof}
(i) Let us consider the case $b_{\Lambda}>0$. For $x\not= b_{\Lambda}$, by differentiating  \eqref{value_func_simplified},
\begin{align}\label{firstder}
v^{b_{\Lambda}\prime}_{\Lambda}(x)&= g_{\Lambda} (b_{\Lambda})\left(W^{(q)\prime}(x)+ \delta \left[\mathbb{W}^{(q)}(0)W^{(q)\prime}(x)+\int_{b_{\Lambda}}^x\mathbb{W}^{(q)\prime}(x-y)W^{(q)\prime}(y)\diff y\right]1_{\{x>b_{\Lambda}\}}\right)\notag\\
&\quad-q\Lambda\left(W^{(q)}(x)+ \delta \left[\mathbb{W}^{(q)}(0)W^{(q)}(x)+  \int_{b_{\Lambda}}^x\mathbb{W}^{(q)\prime}(x-y)W^{(q)}(y)\diff y\right]1_{\{x>b_{\Lambda}\}}\right)\notag\\&\quad -\delta\mathbb W^{(q)}(x-b_{\Lambda})\notag\\
&=g_{\Lambda}(b_{\Lambda})\left(W^{(q)\prime}(x)+\delta\int_{b_{\Lambda}}^x\mathbb{W}^{(q)}(x-y)W^{(q)\prime\prime}(y)\diff y\right)\notag\\
&\quad-q\Lambda\left(W^{(q)}(x)+\delta  \int_{b_{\Lambda}}^x\mathbb{W}^{(q)}(x-y)W^{(q)\prime}(y)\diff y\right)\notag\\
&\quad+\delta\mathbb W^{(q)}(x-b_{\Lambda}) \big(g_{\Lambda}(b_{\Lambda})W^{(q)\prime}(b_{\Lambda})-\Lambda q W^{(q)}(b_{\Lambda})-1 \big),
\end{align}
where the last equality holds by integration by parts. Now by the definition of $g_\Lambda$ as in \eqref{v9.0}, we have
\begin{align}\label{dvf_right}
\begin{split}
v^{b_{\Lambda}\prime}_{\Lambda}(x)	&=g_{\Lambda}(b_{\Lambda})\left(W^{(q)\prime}(x)+\delta\int_{b_{\Lambda}}^x\mathbb{W}^{(q)}(x-y)W^{(q)\prime\prime}(y)\diff y\right) \\
&\quad-q\Lambda\left(W^{(q)}(x)+\delta  \int_{b_{\Lambda}}^x\mathbb{W}^{(q)}(x-y)W^{(q)\prime}(y)\diff y\right).
\end{split}
\end{align}
Differentiating this further, we obtain for $x\not=b_{\Lambda}$
\begin{align}\label{sdvf_right}
\begin{split}
v^{b_{\Lambda}\prime\prime}_{\Lambda}(x)&=g_{\Lambda}(b_{\Lambda})\left((1+\delta\mathbb{W}^{(q)}(0)1_{\{x>b_{\Lambda}\}})W^{(q)\prime\prime}(x)+\delta\int_{b_{\Lambda}}^x\mathbb{W}^{(q)\prime}(x-y)W^{(q)\prime\prime}(y)\diff y\right) \\
&\quad-q\Lambda\left((1+\delta\mathbb{W}^{(q)}(0)1_{\{x>b_{\Lambda}\}})W^{(q)\prime}(x)+\delta  \int_{b_{\Lambda}}^x\mathbb{W}^{(q)\prime}(x-y)W^{(q)\prime}(y)\diff y\right).
\end{split}
\end{align}

By Remark \ref{remark_smoothness_scale_function_SN}, \eqref{dvf_right}, and \eqref{sdvf_right}, the functions $v^{b_{\Lambda}\prime}_{\Lambda}$ and $v^{b_{\Lambda}\prime\prime}_{\Lambda}$ are continuous on $\mathbb{R}\backslash\{b_{\Lambda}\}$. Regarding the continuity at $b_\Lambda$, from \eqref{dvf_right} we have $v^{b_{\Lambda}\prime}_{\Lambda}(b_{\Lambda}+)=v^{b_{\Lambda}\prime}_{\Lambda}(b_{\Lambda}-)$. In particular, for the case that $X$ is of unbounded variation (where $\mathbb{W}^{(q)}(0) = 0$ as in \eqref{eq:Wqp0}), we have, using \eqref{sdvf_right}, that
\begin{align*}
v^{b_{\Lambda}\prime\prime}_{\Lambda}(b_{\Lambda}+)-v^{b_{\Lambda}\prime\prime}_{\Lambda}(b_{\Lambda}-)=0.
\end{align*}
(ii) For the case $b_{\Lambda}=0$, the result follows from a direct application of Lemma \ref{lemma_zero_case}.
\end{proof}

In order to prove the HJB inequality \eqref{HJB-inequality}, we use a more friendly sufficient condition. For the proof of the following result we refer to the proof of Lemma 7 in \cite{KLP}. 

\begin{lemma}\label{tussenlemma} 
The value function $v^{b_{\Lambda}}_{\Lambda}$ satisfies \eqref{HJB-inequality} if and only if
\begin{equation}\label{equiv_inequality2}
\begin{cases}
v^{b_{\Lambda}\prime}_{\Lambda}(x)\geq1, & \text{if $0<x\leq b_{\Lambda}$}, \\
v^{b_{\Lambda}\prime}_{\Lambda}(x)\leq1, & \text{if $x>b_{\Lambda}$}.
\end{cases}
\end{equation}
\end{lemma}
In order to verify inequality \eqref{equiv_inequality2}, we will need the following result.
\begin{lemma}[\cite{Loeffen08}, Theorem 2]\label{compmon}
Under Assumption \ref{assump_completely_monotone}, the $q$-scale function $\mathbb{W}^{(q)}$ can be written as 
\begin{align} \label{W_mathbb_completely_monotone}
	\mathbb{W}^{(q)}(x)=\varphi'(q)e^{\varphi(q)x}-\hat{f}(x),
\end{align}
where $\hat{f}$ is a non-negative, completely monotone function given by $\hat{f}(x)=\displaystyle\int_{0+}^{\infty}e^{-xt}\hat{\mu}(\diff t)$, where $\hat{\mu}$ is a finite measure on $(0,\infty)$. Moreover, $\mathbb{W}^{(q)\prime}$ is strictly log-convex  (and hence convex) on $(0,\infty)$. 
\end{lemma}

\begin{remark}\label{com_mon_forW}
We note that an analogous result to Lemma \ref{compmon} holds for $W^{(q)}$, with $f$ and $\mu$ playing the role of $\hat{f}$ and $\hat{\mu}$.
\end{remark}

We now show the main theorem of this section by verifying the inequality \eqref{equiv_inequality2}. 
\begin{theorem}\label{L.V.1}
The optimal strategy for \eqref{p_Lambda_zero} consists of a refraction strategy at level $b_{\Lambda}$, given by \eqref{b1}, and the corresponding value function is given by \eqref{value_func_simplified}.
\end{theorem}
\begin{proof}
By Lemmas \ref{verificationlemma}, \ref{smooth_fit_prob1}, and \ref{tussenlemma}, it is sufficient to verify \eqref{equiv_inequality2}, and the condition that $v^{b_{\Lambda}}_{\Lambda}(0)\geq -\Lambda$.\\
(i) First, consider the case $b_{\Lambda}>0$. In this case, recall that $\xi_{\Lambda}(b_{\Lambda})=g_{\Lambda}(b_{\Lambda})$ (implied by $\xi'_\Lambda (b_\Lambda) = 0$ and the expression \eqref{v9}).\\
(1) Suppose that $x\leq b_{\Lambda}$. Since $g_{\Lambda}$ is increasing on $(0,a_{\Lambda})$ by Remark \ref{remark_about_g} 
(2) and $b_{\Lambda}\leq a_{\Lambda}$ by Proposition \ref{maxb.1} (1), we obtain
\begin{equation}\label{eq.1}
\xi_{\Lambda}(b_{\Lambda})=g_{\Lambda}(b_{\Lambda})\geq g_{\Lambda}(x)=\dfrac{1+q\Lambda W^{(q)}(x)}{W^{(q)\prime}(x)},\ \text{for all}\ 0 < x\leq b_{\Lambda}. 
\end{equation}
Applying \eqref{eq.1} in \eqref{dvf_right}, it follows that $v^{b_{\Lambda}\prime}_{\Lambda}(x) \geq1$.\\
(2) Now suppose that $x>b_{\Lambda}$.  To show the desired inequality, we first rewrite $v^{b_{\Lambda}\prime}_{\Lambda}$ using the expression of the scale function given in \eqref{W_mathbb_completely_monotone}.  The proof of the following lemma is deferred Appendix \ref{proof_lemma_com_mon}.
\begin{lemma}\label{com_mon}
For $x > b_\Lambda$, we have
\begin{align}\label{form_cmf}
v^{b_{\Lambda}\prime}_{\Lambda}(x)&=\int_{0+}^{\infty}e^{-tx} l(t) \hat{\mu}(\diff t),
\end{align} 
where
\begin{align} \label{def_g_initial}
l(t) &:= g_{\Lambda}(b_{\Lambda})\left((1-\delta W^{(q)}(0))t-\delta t\int_0^{b_{\Lambda}}e^{ty}W^{(q)\prime}(y)\diff y\right)\notag\\
&\quad+q\Lambda\left(1+ \delta t\int_0^{b_{\Lambda}}e^{ty}W^{(q)}(y)\diff y\right)+\delta e^{tb_{\Lambda}}.
\end{align}
\end{lemma}
Differentiating \eqref{def_g_initial} twice, we have
\begin{align}\label{def_g}
l''(t)&= -\delta g_{\Lambda}(b_{\Lambda})\left(2\int_0^{b_{\Lambda}}ye^{ty}W^{(q)\prime}(y)\diff y+t\int_0^{b_{\Lambda}}y^2e^{ty}W^{(q)\prime}(y)\diff y\right)+\delta b_{\Lambda}^2e^{tb_{\Lambda}}\notag\\
&\quad+ \delta q\Lambda\left(2\int_0^{b_{\Lambda}}ye^{ty}W^{(q)}(y)\diff y+t\int_0^{b_{\Lambda}}y^2e^{ty}W^{(q)}(y)\diff y\right).
\end{align}
On the other hand, using \eqref{eq.1} we have that   $g_{\Lambda}(b_{\Lambda})W^{(q)\prime}(y)\geq 1+q\Lambda W^{(q)}(y)$  for all $y\in(0,b_{\Lambda}]$. Hence,
\begin{align*}
l''(t)&\leq  -\delta \left(2\int_0^{b_{\Lambda}}ye^{ty}(1+q\Lambda W^{(q)}(y))\diff y+t\int_0^{b_{\Lambda}}y^2e^{ty}(1+q\Lambda W^{(q)}(y))\diff y\right)+\delta b_{\Lambda}^2e^{tb_{\Lambda}}\notag\\
&\quad+\delta q\Lambda\left(2\int_0^{b_{\Lambda}}ye^{ty}W^{(q)}(y)\diff y+t\int_0^{b_{\Lambda}}y^2e^{ty}W^{(q)}(y)\diff y\right)\\
&=-\delta \biggl(2\int_0^{b_{\Lambda}}ye^{ty}\diff y+t\int_0^{b_{\Lambda}}y^2e^{ty}\diff y\biggr)+\delta b_{\Lambda}^2e^{tb_{\Lambda}}=0.
\end{align*}
Therefore, $l$ is a concave function. In addition, since $l(0)=q\Lambda+\delta$, which is positive by Assumption \ref{assump_SN_param}, and recalling $x > b_\Lambda$, it follows that there exists $0<p\leq\infty$ such that $l$ is positive on $(0,p)$ and negative on $(p,\infty)$. Consequently, 
\begin{align} \label{inequality_l}
\mathrm{e}^{-(x-b_{\Lambda})t}l(t)\geq \mathrm{e}^{-(x-b_{\Lambda})p}l(t), \quad t>0.
\end{align} Now we note from \eqref{def_g_initial} that there exists a constant $C(b_{\Lambda})$  independent of $t$ such that 
$|l(t)|\leq C(b_{\Lambda})(1+t)e^{tb_{\Lambda}}$.
Therefore using the fact that $x>b_{\Lambda}$ and dominated convergence, we can take the derivative inside the integral in \eqref{form_cmf} and obtain
\begin{align}\label{concavity}
v^{b_{\Lambda}\prime\prime}_{\Lambda}(x)&=   -\int_{0+}^{\infty} \mathrm{e}^{-(x-b_{\Lambda})t}\mathrm{e}^{-b_{\Lambda}t}t l(t)\hat{\mu}(\diff t)\\\nonumber
&\leq -\mathrm{e}^{-(x-b_{\Lambda})p}\int_{0+}^{\infty} \mathrm{e}^{-b_{\Lambda}t}tl(t) \hat{\mu}(\diff t) \\\nonumber
&=  \mathrm{e}^{-(x-b_{\Lambda})p} v^{b_{\Lambda}\prime\prime}_{\Lambda}(b_{\Lambda}),
\end{align}
where the inequality holds by \eqref{inequality_l}. On the other hand, Proposition \ref{maxb.1} implies that $b_{\Lambda}\leq a_{\Lambda}$, and hence, by Remark \ref{remark_about_g} (2), 
\begin{align*}
0\leq g'_{\Lambda}(b_{\Lambda})=q\Lambda-\dfrac{W^{(q)\prime\prime}(b_{\Lambda})}{W^{(q)\prime}(b_{\Lambda})}g_{\Lambda}(b_{\Lambda}).
\end{align*}
Therefore, \eqref{sdvf_right} gives
\begin{equation*}
v^{b_{\Lambda}\prime\prime}_{\Lambda}(b_{\Lambda}+) =  (1+\delta \mathbb W^{(q)}(0)) (g_{\Lambda}(b_{\Lambda})W^{(q)\prime\prime}(b_{\Lambda})-q\Lambda W^{(q)\prime}(b_{\Lambda}))\leq 0.
\end{equation*}
In combination with \eqref{concavity}, it follows that $v^{b_{\Lambda}\prime}_{\Lambda}$ is non-increasing on $(b_{\Lambda},\infty)$. On the other hand, we note using \eqref{v9.0} and \eqref{dvf_right}, that
\begin{align*}
v^{b_{\Lambda}\prime}_{\Lambda}(b_{\Lambda})=g_{\Lambda}(b_{\Lambda})W^{(q)\prime}(b_\Lambda)-q\Lambda W^{(q)}(b_\Lambda)=1.
\end{align*}
Hence, we deduce that $v^{b_{\Lambda}\prime}_{\Lambda}(x)\leq 1$ for $x>b_{\Lambda}$.\\
(3) Finally, using \eqref{value_func_simplified}, and the fact that $g_{\Lambda}(b_{\Lambda})\geq0$,
$$v^{b_{\Lambda}}_{\Lambda}(0)=g_{\Lambda}(b_{\Lambda})W^{(q)}(0)-\Lambda\geq -\Lambda,$$
as required.\\
(ii) Now, consider the case $b_{\Lambda}=0$. By taking a derivative in \eqref{v_0_new}, and using \eqref{xi_0_new} and \eqref{W_mathbb_completely_monotone}, we get
\begin{align}\label{excu}
v_{\Lambda}^{0\,\prime \prime}(x)&=(\delta+q\Lambda)\biggr(\frac{\mathbb{W}^{(q) \prime}(x)}{\varphi(q)}- \mathbb{W}^{(q)}(x)\biggl)' \notag\\
&=(\delta+q\Lambda)\Big(-\frac {\hat{f}'(x)} {\varphi(q)} + \hat{f}(x) \Big)' = (\delta+q\Lambda) \Big( -\frac {\hat{f}''(x)}{\varphi(q)} + \hat{f}'(x) \Big),
\end{align}

which is negative because $\hat{f}$ is completely monotone. Therefore $v_{\Lambda}^{0\,\prime}(x)$ is non-increasing, and hence it is enough to verify that $v_{\Lambda}^{0\,\prime}(0+)\leq1$ or equivalently
\begin{equation*}
\frac{(\delta+q\Lambda)\mathbb{W}^{(q)\prime}(0+)}{1+(\delta+q\Lambda)\mathbb{W}^{(q)}(0)}\leq\varphi(q).
\end{equation*}
This inequality is automatically satisfied in cases (i) and (ii) given in Proposition \ref{maxb.1} (2). Therefore, we have \eqref{equiv_inequality2} when $b_{\Lambda}=0$. To finish the proof, using \eqref{xi_0_new}, \eqref{v_0_new} and Assumption \ref{assump_SN_param}, we have
$$v_{\Lambda}^{0}(0)=\frac{(\delta+q\Lambda)}{\varphi(q)}\mathbb{W}^{(q)}(0)-\Lambda\geq-\Lambda.$$				
\end{proof}

\section{Solution of the constrained de Finetti's problem}\label{const.1}

In this section, we study the constrained de Finetti's problem given in \eqref{F.1}  under Assumptions  \ref{assump_completely_monotone} and  \ref{assumpdelta}. In order to solve this problem, we use the results in Section \ref{optimal_dividend_tv_sn}, noting that the optimal strategy for \eqref{p_Lambda} for any $K\in[0,1]$ is the same as the case $K=0$, i.e., 
 $D^{b_{\Lambda}}$, with $b_{\Lambda}$ as in  \eqref{b1}, is the optimal strategy for \eqref{p_Lambda}.  See the discussion at the end of Section \ref{subsection_constrained_de_finetti}. 
 
Throughout this section, we assume the following (see Remark \ref{rem.1} for the case it does not hold).
 \begin{assump} \label{assump_avoid_all_zero}
 We assume that $\bar{\Lambda} < \infty$, which, by Remark \ref{remark_b_zero_all}, is equivalent  to
 \begin{equation}\label{as.1}
\varphi(q)< \phi_1(\infty) = \dfrac{q+\Pi(0,\infty)}{c-\delta},\ \text{if}\  \sigma=0 \ \text{and}\ \Pi(0,\infty)<\infty.
\end{equation}
 \end{assump}
  
First we need to study the relationship between $\Lambda$ and its corresponding optimal barrier level $b_{\Lambda}$, which will give us enough tools to see if the problem \eqref{F.1} is feasible or not.



Recall Remark \ref{remark_b_zero_all} and fix $\Lambda > \bar{\Lambda}$ (then necessarily $b_\Lambda > 0$).
Since $\xi'_{\Lambda}(b_{\Lambda})=0$ and by the first equality of \eqref{v9}, we observe that $\Lambda$ and $b_\Lambda$ satisfy the relation $\Lambda = \lambda(b_\Lambda)$ where 
\begin{align}\label{L1}
\lambda(b) := ( q  H(b) )^{-1},  
\end{align}
with
\begin{align}\label{rem.4}
H(b) := \frac {[h(b)]^{2}} {h'(b)}- \bigg( W^{(q)}(b)+ \frac {h(b)} {\varphi(q)} \bigg).
\end{align}

\begin{lemma} \label{lemma_H_positive}
The function $H(b)$ given in \eqref{rem.4} is positive for $b\in(b_0,\infty)$, where we recall that $b_{0}$ is as in \eqref{b1} when $\Lambda=0$. 
\end{lemma}
\begin{proof}
First we note that, using \eqref{funh},
	\begin{align}
		h'(b)=\varphi(q)(h(b)-W^{(q)\prime}(b)). \label{h_diff}
	\end{align}
Hence, 
\begin{equation}\label{L5}
h'(b)H(b)=\varphi(q)W^{(q)}(b)\bigg(\dfrac{W^{(q)\prime}(b)}{W^{(q)}(b)}\bigg(W^{(q)}(b)+\dfrac{h(b)}{\varphi(q)}\bigg)-h(b)\bigg).
\end{equation}
Now since $h'>0$ on $(b_{0},\infty)$ (see Remark \ref{remfunh}(i)),
it is enough to show that the right side of \eqref{L5} is positive. On the other hand, we know from \cite{HJM17} that $W^{(q)}$ is log-concave on $(0,a_0]$ and strictly log-concave on $(a_0,\infty)$, where $a_0$ is defined in Remark \ref{remark_about_g} for the case $\Lambda = 0$. Then,
\begin{equation*}
	\dfrac{W^{(q)\prime}(\eta)}{W^{(q)}(\eta)}\geq\dfrac{W^{(q)\prime}(\varsigma)}{W^{(q)}(\varsigma)},\ \text{for any $\eta$ and $\varsigma$ with $0 <  \eta\leq\varsigma$.}
\end{equation*}
Note that the previous inequality is strict when $a_0 <\eta<\varsigma$. From here, it can be verified that
\begin{equation*}
	\dfrac{W^{(q)\prime}(b)}{W^{(q)}(b)}\int_{b}^{\infty}e^{-\varphi(q)y} W^{(q)}(y)\diff y> \int_{b}^{\infty}e^{-\varphi(q)y} W^{(q)\prime}(y) \diff y,
\end{equation*}
and using \eqref{funh} and \eqref{int_by_parts_laplace}, it follows that 
\begin{equation}\label{l.5}
\dfrac{W^{(q)\prime}(b)}{W^{(q)}(b)}\bigg(W^{(q)}(b)+\dfrac{h(b)}{\varphi(q)}\bigg)>h(b),\ \text{for all $b>0$}.
\end{equation}
From  \eqref{L5} and \eqref{l.5}, we have  $h'(b)H(b)>0$ for $b \in (b_0,\infty)$, as desired.
\end{proof}
Note that Lemma \ref{lemma_H_positive} implies that $\lambda(b)$ is finite and positive for $b\in(b_0,\infty)$.

\begin{proposition}\label{Lam.1}
Assume that \eqref{as.1} holds. Then, the function $\lambda(b)$, given in \eqref{L1}, is 
	(i)  strictly increasing for all $b>b_0$, 
	 (ii)  $\displaystyle\lim_{b\downarrow b_{0}}\lambda(b)=\bar{\Lambda} \vee 0$, (iii)  $\displaystyle\lim_{b\rightarrow \infty}\lambda(b)=\infty$ and (iv)  $b_{\lambda(b)}=b$ for all $b>b_{0}$.  
\end{proposition}
\begin{proof}
(i) Taking the first derivative in \eqref{rem.4} and by  \eqref{h_diff}, we have that
\begin{align*}
H^{\prime}(b)
&= h(b)-\dfrac{[h(b)]^{2}h''(b)}{[h'(b)]^{2}}\notag\\
&=\dfrac{h(b)}{[h'(b)]^{2}}([h'(b)]^{2}-h(b)h''(b))\notag\\
&=\dfrac{[\varphi(q)h(b)]^{2}W^{(q)\prime}(b)}{[h'(b)]^{2}}\biggr(\dfrac{W^{(q)\prime}(b)}{h(b)}+\dfrac{W^{(q) \prime \prime}(b)}{\varphi(q)W^{(q)\prime}(b)}-1\biggr).
\end{align*}
where the last equality holds because
\begin{align*}
[h'(b)]^{2}-h(b)h''(b) &= \varphi^2(q)(h(b)-W^{(q)\prime}(b))^2 - h(b) \varphi(q)[\varphi(q)(h(b)-W^{(q)\prime}(b))-W^{(q)\prime \prime}(b)] \\
&=  \varphi(q) \Big\{ \varphi(q) ( W^{(q)\prime}(b)^2 - h(b) W^{(q)\prime}(b) ) + h(b) W^{(q)\prime \prime}(b) \Big\}.
\end{align*}
Since by Remark \ref{com_mon_forW} $W^{(q)\prime}$ is a strictly log-convex function, we have that 
\begin{equation*}
\dfrac{W^{(q)\prime\prime}(\eta)}{W^{(q)\prime}(\eta)}<\dfrac{W^{(q)\prime\prime}(\varsigma)}{W^{(q)\prime}(\varsigma)},\ \text{for any $\eta$ and $\varsigma$ with $0 < \eta<\varsigma$.}
\end{equation*}
From the above and integration by parts we can show that
\begin{align*}
\frac{W^{(q)\prime\prime}(b)}{W^{(q)\prime}(b)}h(b)&<\varphi(q)e^{\varphi(q)b}\int_b^\infty e^{-\varphi(q) y} W^{(q)\prime\prime} (y) \diff  y=-\varphi(q)W^{(q)\prime}(b)+\varphi(q)h(b),
\end{align*}
and hence
\begin{align*}
\dfrac{W^{(q)\prime}(b)}{h(b)}+\dfrac{W^{(q)''}(b)}{\varphi(q)W^{(q)\prime}(b)}-1<\dfrac{W^{(q)\prime}(b)}{h(b)}+\dfrac{h(b)-W^{(q)\prime}(b)}{h(b)}-1=0.
\end{align*}
Hence, we conclude that the function $H$ as in \eqref{rem.4} is strictly decreasing or equivalently $\lambda$ is strictly increasing.\\
(ii) For the case $\bar{\Lambda} \geq 0$, Remark \ref{remark_b_zero_all} gives
$b_{0}=0$. Then, from \eqref{xi_0_new}, \eqref{rem.4} and \eqref{h_diff}
\begin{align}
\lim_{b\rightarrow0}H(b)&=\frac{[h(0)]^2}{h'(0+)}-\left(W^{(q)}(0)+\frac{h(0)}{\varphi(q)}\right)\notag\\
&=\frac{h(0)W^{(q)\prime}(0+)-\varphi(q)W^{(q)}(0)(h(0)-W^{(q)\prime}(0+))}{\varphi(q)(h(0)-W^{(q)\prime}(0+))}\notag\\
&=\dfrac{\delta^{-1}W^{(q)\prime}(0+)-\varphi(q)W^{(q)}(0)(\delta^{-1}-W^{(q)}(0))}{\varphi(q)(\delta^{-1}-W^{(q)}(0))-W^{(q)\prime}(0+)}.
\end{align}
Using \eqref{eq:Wqp0}, \eqref{W_zero_derivative} and the fact that $\phi_1(\bar{\Lambda})=\varphi(q)$ or   $\phi_2(\bar{\Lambda})=\varphi(q)$ (see Remark \ref{remark_b_zero_all}), it can be verified that $\displaystyle\lim_{b\rightarrow0}H(b)=1/q\bar{\Lambda}$ (where in the case $\bar{\Lambda}=0$ the right-hand side is understood to be infinity), and hence  $\displaystyle\lim_{b\downarrow b_{0}} \lambda(b) = \lim_{b\downarrow 0} \lambda(b)=\bar{\Lambda}$. 

For the case $\bar{\Lambda} < 0$, Remark \ref{remark_b_zero_all} gives $b_{0}>0$. By Lemma 3 of \cite{KLP}, we know that $h$ attains its unique minimum at $b_{0}$ and by the continuity of $h'$, $\displaystyle\lim_{b\rightarrow b_{0}}h'(b)=h'(b_{0})=0$. In addition, by \eqref{h_diff}, $\displaystyle\lim_{b\rightarrow b_{0}}h(b) = W^{(q)\prime} (b_0) > 0$. Therefore, from \eqref{rem.4}, we get that $\displaystyle\lim_{b\rightarrow b_{0}}H(b)=\infty$, and hence $\displaystyle\lim_{b\downarrow b_{0}}\lambda(b)=0 =\bar{\Lambda} \vee 0$.\\
(iii) From Remark \ref{com_mon_forW}, we can write
\begin{align*}
h(b) = \varphi(q) \bigg(\frac { \Phi(q) \Phi'(q) e^{\Phi(q) b}} {\varphi(q) - \Phi(q)} - \tilde{f}(b) \bigg),
\end{align*}
where $\tilde{f}(b) := \displaystyle\int_{0}^{\infty}e^{-\varphi(q)y}f'(y+b)\diff y$, and hence
we get the following expressions:
\begin{align*}
[h(b)]^{2} &= [\varphi(q)]^{2}\bigg(\dfrac{[\Phi(q)\Phi'(q)]^{2}}{(\varphi(q)-\Phi(q))^{2}}e^{2\Phi(q)b}-\frac{2\Phi(q)\Phi'(q)}{\varphi(q)-\Phi(q)}e^{\Phi(q)b}\tilde{f}(b)+[\tilde{f}(b)]^{2}\bigg),\\
h'(b) &=\varphi(q)\biggl(\dfrac{[\Phi(q)]^{2}\Phi'(q)}{\varphi(q)-\Phi(q)}e^{\Phi(q)b}-\varphi(q)\tilde{f}(b)+f'(b)\biggr),
\end{align*}
and
\begin{equation*}
W^{(q)}(b)+\frac{h(b)}{\varphi(q)} =\dfrac{\Phi'(q)\varphi(q)}{\varphi(q)-\Phi(q)}e^{\Phi(q)b}-\tilde{f}(b)-f(b).
\end{equation*}
  Applying these identities in \eqref{rem.4}, it follows that 
\begin{align*}
H(b)&=\dfrac{\varphi(q)\Phi(q)\Phi'(q)}{\varphi(q)-\Phi(q)} H_1(b) +\dfrac{\varphi(q)^2[\tilde{f}(b)]^{2}}{h'(b)}+\tilde{f}(b)+f(b),
\end{align*}
where
\begin{align*}
H_1(b) &:=e^{\Phi(q)b}  \left[\frac {\varphi(q)} {h'(b)} \left( \dfrac{ \Phi(q)\Phi'(q)}{\varphi(q)-\Phi(q)}e^{\Phi(q)b}-2\tilde{f}(b) \right) -\dfrac{1}{\Phi(q)}\right] \\
&= \frac {\varphi(q)} {\Phi(q)}
\dfrac{\varphi(q)\tilde{f}(b)-2\Phi(q)\tilde{f}(b)-f'(b)}{h'(b)}.
\end{align*}
By the dominated convergence theorem we have that $f(b)\to 0$, $\tilde{f}(b)\to 0$, and $f'(b)\to 0$ as $b\to\infty$. Hence $\displaystyle\lim_{b\rightarrow \infty}H(b)=0$ or equivalently
$\displaystyle\lim_{b\rightarrow \infty}\lambda(b)=\infty$.\\
(iv) Fix $b > b_0$. First let us assume that $b_0>0$ or that $b_0=0$ and $h'(0+)=0$. Then using \eqref{v9} for $\Lambda = \lambda(b)$, we have
\begin{align}
\dfrac{\diff\xi_{\lambda(b)}(\varsigma)}{\diff\varsigma} & = q \lambda (b)-\dfrac{h'(\varsigma)}{h(\varsigma)}\xi_{\lambda(b)}(\varsigma),\qquad\text{for $\varsigma > 0$}. \label{xi_derivative_b}
\end{align}
For $\varsigma\in (0,b_0]$, by the fact that $h'(\varsigma)\leq0$ for $\varsigma\in[0,b_0]$ (see Remark \ref{remfunh}(i)) we have that $\dfrac{\diff\xi_{\lambda(b)}(\varsigma)}{\diff\varsigma} > 0$. On the other hand, applying \eqref{h_diff} and  \eqref{v5} in \eqref{xi_derivative_b},
\begin{align}\label{eq.2}
\dfrac{\diff\xi_{\lambda(b)}(\varsigma)}{\diff\varsigma}&=\dfrac{1}{[h(\varsigma)]^{2}}\biggr(q\lambda(b)\biggr([h(\varsigma)]^{2}-h'(\varsigma)\biggr(W^{(q)}(\varsigma)+\dfrac{h(\varsigma)}{\varphi(q)}\biggl)\biggr)-h'(\varsigma)\biggl)\notag\\
&=\dfrac{h'(\varsigma)}{[h(\varsigma)]^{2}}\biggr(\frac{\lambda(b)}{\lambda(\varsigma)}-1\biggl).
\end{align}
Then using the fact that $\varsigma\mapsto\lambda(\varsigma)$ is strictly increasing as in (i), and that $h'(\varsigma)>0$ for $\varsigma>b_0$ (see Remark \ref{remfunh}(i)) we have that  $\dfrac{\diff\xi_{\lambda(b)}(\varsigma)}{\diff\varsigma}>0$ for $b_0<\varsigma<b$ and vanishes at $\varsigma=b$. Now let us assume that $b_0=0$ and that $h'(0+)>0$. Then, by the proof of Lemma 3 in \cite{KLP}, we have that $h'(\zeta)>0$ for all $\zeta >0$. Hence \eqref{eq.2} implies that $\dfrac{\diff\xi_{\lambda(b)}(\varsigma)}{\diff\varsigma}>0$ for $0<\varsigma<b$  and vanishes at $\varsigma=b$. The above implies that $b=\inf\left\{\varsigma\geq0:\dfrac{\diff\xi_{\lambda(b)}(\varsigma)}{\diff\varsigma}\leq0\right\}$.
Therefore by \eqref{b1} we obtain that $b_{\lambda(b)}=b$ for all $b>b_{0}$.
\end{proof}

Now, by \eqref{lap_ac} and \eqref{int_by_parts_laplace}, we see that
\begin{align}\label{lap_ac.1}
\begin{split}
\Psi_{x}(b)&=Z^{(q)}(x)+\delta q \int_b^x\mathbb{W}^{(q)}(x-y)W^{(q)}(y)\diff y \\
&\quad-\frac{q(W^{(q)}(b)+h(b)/\varphi(q))}{h(b)}\left(W^{(q)}(x)+\delta\int_b^x\mathbb{W}^{(q)}(x-y)W^{(q)\prime}(y)\diff y\right),
\end{split}
\end{align}
for all $b\in[0,\infty)$. 
\begin{remark} \label{remark_x_zero} 
Note that if $x=0$ and $X$ is of unbounded variation, we immediately obtain that $\Psi_{0}(b)=1$, for all $b\geq0$. Hence we obtain that $V(0; K)=0$ if $K=1$, and the problem is infeasible otherwise.
\end{remark}
The proof of the following result is deferred to Appendix \ref{appendix_decPsc1}.
\begin{lemma}\label{decPsi1}
Assume $x\geq0$ and in the case $X$ is of unbounded variation that $x>0$.  Then the function $b \mapsto \Psi_{x}(b)$ defined in \eqref{lap_ac.1} is strictly decreasing on $[0,\infty)$ and 
\begin{equation}\label{L.14}
K_{x}:=\lim_{b\rightarrow\infty}\Psi_{x}(b)=Z^{(q)}(x)-\dfrac{q}{\Phi(q)}W^{(q)}(x).
\end{equation} 
\end{lemma}

By Remark \ref{remfunh}(ii), the limit of  \eqref{div_ac} becomes
\begin{align}\label{div_b_to_inf}
\lim_{b\to\infty}\mathbb{E}_x\left[\int_0^{\tau_{b}} e^{-qt}\diff D^b_t\right]=0, \quad x \geq 0.
\end{align}
	
We will denote, for $K \geq 0$,
		\begin{equation*}
		v_{\Lambda}^b(x;K):=v^b_{\Lambda}(x)+\Lambda K.
	\end{equation*}
Therefore, using \eqref{div_b_to_inf} and Lemma \ref{decPsi1} in \eqref{vf_1_new} we obtain that
\begin{align}\label{vf_1_limit}
\lim_{b\to\infty}v_{\Lambda}^b(x;K)=\Lambda(K- K_x), \qquad\text{for all $x>0$.}
\end{align}

We are now ready to characterize the solution of \eqref{F.1}. We define  the do-nothing strategy as  $D^{\infty}=0$ and hence $U_t^{D^\infty} = X$. By (8.9) of \cite{kyprianou2014}, \eqref{vf_1_limit} and Lemma \ref{decPsi1}, we confirm the following convergence results:
\begin{equation}\label{inf.1}
\E_{x}\Bigr[e^{-q\tau^{D^{\infty}}}\Bigr]=K_{x}=\lim_{b\to\infty}\Psi_x(b)
\quad \text{and}\quad v_{\Lambda}^{D^{\infty}}(x;K)=\Lambda(K-K_{x})=\lim_{b\to\infty}v_{\Lambda}^{D^{b}}(x;K),
\end{equation}
where $\tau^{D^{\infty}}:= \inf \{t>0:X_{t}<0\}$. 

From \eqref{p_dual} and \eqref{inf.1}, we observe  that, if $K \geq K_x$, then $D^\infty$ is feasible for the problem \eqref{F.1} and hence
\begin{equation}\label{inf.2}
V(x;K) =\sup_{D\in\Theta}\inf_{\Lambda\geq0}v_{\Lambda}^{D}(x; K)\geq\inf_{\Lambda\geq0}v_{\Lambda}^{D^{\infty}}(x;K)=0,\ \text{if}\ K\in[K_{x},1],
\end{equation}
with $v_{\Lambda}^{D}$ as in \eqref{lm}.

\begin{theorem}\label{main.1}
Let $x\geq0$ be fixed. Assume \eqref{as.1}  and one of the following cases: (1) $x>0$ and $X$ is of unbounded variation; (2) $x\geq0$ and $X$ is of bounded variation. Then,
\begin{equation*}
V(x;K)=\begin{cases}
v^{{b_{0}}}_{0}(x;K),&\text{if}\ K\in\left[\Psi_{x}(b_{0	}),1\right],\\
\displaystyle\inf_{\Lambda\geq0}V_{\Lambda}(x;K),&\text{if}\ K\in(K_{x},\Psi_{x}(b_{0})),\\
0,&\text{if}\ K=K_{x},\\
-\infty,&\text{if}\ K\in[0,K_{x}).
\end{cases}
\end{equation*} 	
\end{theorem}

\begin{proof}
We will only prove $(1)$ since the other case is similar. 
Recall inequality \eqref{p_dual} and that $V_{\Lambda}(x;K)$ is defined as in \eqref{p_Lambda}.\\
(i) If $K\in[\Psi_{x}(b_{0}),1]$, then the threshold strategy at level $b_0$ is feasible for the problem \eqref{F.1} (see Section \ref{formProblem}), and therefore
\begin{equation*}
v^{{b_{0}}}_{0}(x;K)\leq V(x;K)\leq\displaystyle\inf_{\Lambda\geq0}V_{\Lambda}(x;K)\leq V_0(x;K)=v^{{b_{0}}}_{0}(x;K). 
\end{equation*}
Here the second inequality holds by \eqref{p_dual} and the last equality holds because the case $\Lambda = 0$ is solved by the threshold strategy with $b_0$ in the problem \eqref{p_Lambda} (which is equivalent to \eqref{p_Lambda_zero}).\\
(ii) If $K\in(K_{x},\Psi_{x}(b_{0}))$, since $b \mapsto\Psi_{x}(b)$ is continuous and strictly decreasing, by Lemma \ref{decPsi1}, to $K_x$, there exists a unique $b^{*}>b_{0}$ such that $K-\Psi_{x}(b^{*})=0$.  Therefore, 
\begin{equation*}
V(x;K)\leq\inf_{\Lambda\geq0}V_{\Lambda}(x;K)\leq V_{\lambda(b^*)}(x;K)=\mathbb{E}_x\left[\int_0^{\tau_{b^*}} e^{-qt}\diff D^{b^{*}}_t\right] \leq V(x;K).
\end{equation*}
Here the first inequality holds by \eqref{p_dual}. The equality follows from Proposition \ref{Lam.1}(iv) since $D^{b^{*}}$ is the optimal strategy for \eqref{p_Lambda} when $\Lambda=\lambda(b^{*})$. The last inequality follows since the threshold strategy at level $b^*$ is feasible for the problem \eqref{F.1}. \\
(iii) If $K=K_{x}$, by Lemma \ref{decPsi1}, we have that $\lambda(b)(K-\Psi_{x}(b))\leq0$ for all $b> b_{0}$. Hence,
\begin{equation*}
0\leq V(x;K)\leq\inf_{\Lambda\geq0}V_{\Lambda}(x;K)\leq
\inf_{b>b_0} \bigg(\mathbb{E}_x\left[\int_0^{\tau_{b}} e^{-qt}\diff D^b_t\right]+\lambda(b)(K-\Psi_{x}(b))\bigg)\leq 0.
\end{equation*}
Here the first inequality follows from \eqref{inf.2}, second by \eqref{p_dual}, and the last one by \eqref{div_b_to_inf}.\\
(iv) Finally, if $K\in[0,K_{x})$, then Lemma \ref{decPsi1} gives $\displaystyle\lim_{b\rightarrow\infty}[K-\Psi_{x}(b)]=K-K_{x}<0$. Hence, we get that  
\begin{equation*}
V(x;K)\leq\displaystyle\inf_{\Lambda\geq0}V_{\Lambda}(x;K)\leq
\inf_{b>b_0}  \bigg(\mathbb{E}_x\left[\int_0^{\tau_{b}} e^{-qt}\diff D^b_t\right]+\lambda(b)(K-\Psi_{x}(b))\bigg)=-\infty,
	\end{equation*}
where the second inequality holds by \eqref{p_dual} and the last equality follows from Proposition \ref{Lam.1}.
\end{proof}
\begin{remark}\label{rem.1}
	Consider the case  Assumption \ref{assump_avoid_all_zero} is violated. By Remark \ref{remark_b_zero_all},  for all $\Lambda \geq 0$, we must have $b_\Lambda = 0$ and hence $V_{\Lambda}(x;K)=\mathbb{E}_x\left[\displaystyle\int_0^{\tau_{0}} e^{-qt}
	\diff D^{0}_{t}\right]+\Lambda(K-\Psi_{x}(0))$. If $K\in[\Psi_{x}(b_{0}),1]$, then the threshold strategy at level $0$ is feasible for the problem \eqref{F.1} (see Section \ref{formProblem}), and therefore
	\begin{equation*}
		v^{0}_{0}(x;K)\leq V(x;K)\leq\displaystyle\inf_{\Lambda\geq0}V_{\Lambda}(x;K)\leq V_0(x;K)=v^{0}_{0}(x;K). 
	\end{equation*}	
On the other hand, if $K\in[0,\Psi_{x}(0))$, we obtain
	\begin{equation*}
		V(x;K)\leq \inf_{\Lambda\geq0}V_{\Lambda}(x;K)=\inf_{\Lambda\geq0}\left(\mathbb{E}_x\left[\displaystyle\int_0^{\tau_{0}} e^{-qt}
		\diff D^{0}_{t}\right]+\Lambda(K-\Psi_{x}(0))\right)=-\infty.	
	\end{equation*}
	In sum, we have
	\begin{equation*}
		V(x;K)=
		\begin{cases}
			v^{0}_{0}(x,K),&\text{if}\ K\in[\Psi_{x}(0),1],\\
			-\infty,&\text{if}\ K\in[0,\Psi_{x}(0)).
		\end{cases}
	\end{equation*}
\end{remark}

\section{Spectrally positive case}\label{specpos}

In this section, we solve analogous problems driven by a spectrally positive \lev process $\overline{Y}$.  We assume that its dual process $Y = -\overline{Y}$ has its Laplace exponent $\psi_Y$ as in  \eqref{lap_exp_Y} so that its right inverse and scale function are given by $\varphi(q)$ and $\mathbb{W}^{(q)}$, respectively.  We also define the drift-changed process $\overline{X} =\{\overline{X}_t = \overline{Y}_t - \delta t; t \geq 0 \}$ whose dual $X = - \overline{X}$ has its Laplace exponent $\psi$ as in \eqref{laplace_exponent}, right inverse $\Phi(q)$ and scale function $W^{(q)}$ described in Section  \ref{section_scale_functions}. We denote by $\overline{\E}_x$ the expectation with respect to the law of the process $\overline{Y}$  when it starts at $x$.  
In addition, for $x\geq0$ and $0< K\leq 1$, we define $\overline{V}(x;K)$, $\overline{v}_{\Lambda}^D(x;K)$, $\overline{V}_\Lambda(x;K)$ and  $\overline{v}_\Lambda(x)$ analogously to \eqref{F.1},  \eqref{lm}, \eqref{p_Lambda} and \eqref{V_Lambda_def}, respectively.

We first solve the optimal dividend problem with terminal payoff/penalty \eqref{p_Lambda_zero} with $X$ replaced with $\overline{Y}$. Similarly to the spectrally negative L\'evy case, we define for $b \geq 0$ the threshold strategy $\overline{D}^b$ and the resulting controlled surplus process, which is a refracted spectrally positive L\'evy process defined as the unique strong solution to the stochastic differential equation,
\begin{equation}\label{re.1}
\overline{U}^b_{t}:=\overline{Y}_{t}-\overline{D}^b_t:=\overline{Y}_{t}-\delta\int_{0}^{t}1_{\{\overline{U}^b_{s}>b\}}\diff s,\ t\geq0.
\end{equation}
Let its ruin time be denoted by
\begin{align*}
\overline{\tau}_{b}:=\inf\{t>0: \overline{U}^{b}_t < 0 \}.
\end{align*}

\subsection{Scale functions under a change of measure.} \label{section_change_of_measure}
For each $ \beta \geq0$, we define the change of measure
\begin{equation}\label{change_meas}
\dfrac{\diff\tilde{\mathbb{P}}^{\beta}_{x}}{\diff\tilde{\mathbb{P}}_{x}}\biggr|_{\mathcal{F}_{t}}=e^{\beta(Y_{t}-x)-\psi_Y(\beta)t},\ x\in\R, t\geq0,
\end{equation}
where $\tilde{\mathbb{P}}_x$ is law of the process $Y$ when it starts at $x$. It is known that $Y$ is still a spectrally negative L\'evy process on  $(\Omega, \mathcal{F}, \tilde{\mathbb{P}}^{\beta})$ and the scale function of $Y$ on this probability space can be written
\begin{equation}\label{aux}
\begin{split}
\mathbb{W}_{ \beta }^{(u-\psi_Y( \beta ))}(x)&=e^{-\beta x}\mathbb{W}^{(u)}(x), \\
\mathbb{Z}_{ \beta}^{(u-\psi_{Y}(\beta))}(x)&=1+(u-\psi_Y( \beta ))\int_{0}^{x}e^{- \beta z}\mathbb{W}^{(u)}(z)\diff z,
\end{split}
\end{equation}
with $u-\psi_Y( \beta)\geq0$; see \cite[Remark 4]{AvPaPi07}. 
In particular, \eqref{lap_exp_Y} and \eqref{def_varphi} give $q-\psi_{Y}(\Phi(q))=\delta\Phi(q)$ and hence 
\begin{align}\label{wchange}
	\mathbb{W}^{(\delta\Phi(q))}_{\Phi(q)}(x)=e^{-\Phi(q)x}\mathbb{W}^{(q)}(x) \quad \text{and}\quad \mathbb{Z}^{(\delta\Phi(q))}_{\Phi(q)}(x)= 1+\delta\Phi(q)\int_{0}^{x}e^{-\Phi(q)z}\mathbb{W}^{(q)}(z)\diff z.
\end{align}

\subsection{Optimal dividend problem with terminal value} \label{problem_terminal_value_dual}
As in the case of  spectrally negative L\'evy process, we are  first interested in solving the problem \eqref{p_dual} for the spectrally positive case. For this purpose, first we need to study the optimal dividend problem with a terminal value for the process $\overline{Y}$. Using Theorems 5.(i) and 6.(iii) in \cite{KyLo} we have the following result, whose proof is deferred to  Appendix \ref{proof_lemma_r.1}.

\begin{proposition}\label{r.1}
For $x,b,q\geq0$, we have
\begin{align}
\overline{\Psi}_{x}(b):=\overline{\E}_x\left[e^{-q\overline{\tau}_{b}};\overline{\tau}_{b}<\infty\right]&=e^{-\Phi(q)x}\frac{\mathbb{Z}^{(\delta\Phi(q))}_{\Phi(q)}(b-x)}{\mathbb{Z}^{(\delta\Phi(q))}_{\Phi(q)}(b)},\label{d.1}\\
\intertext{where $\mathbb{Z}^{(\delta\Phi(q))}_{\Phi(q)}(x)$ is given in \eqref{wchange}, $\overline{\tau}_{b}:=\inf\{ t>0: \overline{U}^b_t=0 \}$ and}
\overline{\E}_x\biggr[\displaystyle \int_0^{\overline{\tau}_{b}} e^{-qt}\diff \overline{D}^{b}_{t}\biggl]&=\dfrac{\delta}{q} \Big(\mathbb{Z}^{(q)}(b-x)-\mathbb{Z}^{(q)}(b)\overline{\Psi}_{x}(b) \Big).\label{d.2}
\end{align}
\end{proposition}
 Using \eqref{d.1} and \eqref{d.2}, we have the following result.

\begin{proposition}
	 For $b \geq 0$, we have
	\begin{equation}\label{d.3}
	\overline{v}_{\Lambda}^{b}(x):=\overline{\E}_x\left[\int_0^{\overline{\tau}_{b}} e^{-qt}\diff \overline{D}^b_t\right]- \Lambda \overline{\Psi}_{x}(b) =
	\begin{cases}
	\dfrac{\delta}{q}\mathbb{Z}^{(q)}(b-x)-\overline{k}_{x}(b,\Lambda),&\text{if}\ x\leq b,\\
	\vspace{-0.6cm}\\
	\dfrac{\delta}{q}-\overline{k}_{b}(b,\Lambda)e^{-\Phi(q)(x-b)},&\text{if}\ x>b,
	\end{cases}
	\end{equation}
	where for $x\geq0$,
	\begin{align*}
		\overline{k}_{x}(b,\Lambda)&:=\overline{\Psi}_{x}(b)\biggr(\dfrac{\delta}{q}\mathbb{Z}^{(q)}(b)+\Lambda\biggl).
	\end{align*}
\end{proposition}
In order to select the optimal barrier, we apply smooth fit. Note that, by \eqref{d.3}, $\overline{v}_{\Lambda}^{b}$ is continuous on $[0,\infty)$ for any choice of $b$. Here, we will study the smoothness of  $\overline{v}_{\Lambda}^{b}$ at $x=b$  to propose a candidate barrier level $\overline{b}_{\Lambda}$ such that $\overline{v}_{\Lambda}^{\overline{b}_{\Lambda}}$ is $\hol^{1}(0,\infty)$ and $\hol^{2}(0,\infty)$ when $\overline{Y}$ is of bounded and unbounded variation, respectively. By differentiating \eqref{d.3},  we see that
\begin{equation}\label{d.4}
\overline{v}_{\Lambda}^{b\prime}(x)=
\begin{cases}
-\delta\mathbb{W}^{(q)}(b-x)+\Phi(q) (\overline{k}_{x}(  b ,\Lambda)+\delta\overline{k}_{b}(b,\Lambda)\mathbb{W}^{(q)}(b-x)),&\text{if}\ x<b,\\
\Phi(q)\overline{k}_{b}(b,\Lambda)e^{-\Phi(q)(x-b)},&\text{if}\ x>b,
\end{cases}
\end{equation}
and for the unbounded variation case 
\begin{equation}\label{d.5}
\overline{v}_{\Lambda}^{b\prime\prime}(x)=
\begin{cases}
\delta\mathbb{W}^{(q)\prime}(b-x)-[\Phi(q)]^{2} \overline{k}_{x}( b ,\Lambda)&\\
\hspace{2.42cm}-\delta\Phi(q)\overline{k}_{b}(b,\Lambda)[\Phi(q)\mathbb{W}^{(q)}(b-x)+\mathbb{W}^{(q)\prime}(b-x)],&\text{if}\ x<b,\\
-[\Phi(q)]^{2}\overline{k}_{b}(b,\Lambda)e^{-\Phi(q)(x-b)},&\text{if}\ x>b,
\end{cases}
\end{equation}
where we recall that, if $Y$ is of unbounded variation,  $\mathbb{W}^{(q)}$ is $\hol^{1}$ on $(0,\infty)$. From \eqref{d.4}  and \eqref{d.5}, together with \eqref{eq:Wqp0} and \eqref{W_zero_derivative}, we have the following result.
\begin{lemma} \label{smoothnes_sp}
		Suppose $b > 0$ is such that
		\begin{equation}\label{d.8.0}
		\overline{k}_{b}(b,\Lambda)=\dfrac{1}{\Phi(q)},\quad\text{or equivalently}\quad \Lambda e^{-\Phi(q)b} = s(b),
		\end{equation}
		where
			\begin{equation}\label{d.8.1}
			s(b):=\frac{1}{\Phi(q)}\mathbb{Z}^{(\delta\Phi(q))}_{\Phi(q)}(b)-\dfrac{\delta e^{-\Phi(q)b}}{q}\mathbb{Z}^{(q)}(b),\ b> 0.
			\end{equation}
			Then, the function $\overline{v}_{\Lambda}^{b}$ is $C^1(0, \infty)$ and $C^2(0, \infty)$ for the case of bounded and unbounded variation, respectively.
	\end{lemma}

\begin{lemma}\label{opt.1}
If $\Lambda>\dfrac{1}{\Phi(q)}-\dfrac{\delta}{q}$, then there exists a unique $b>0$ that satisfies \eqref{d.8.0}.
\end{lemma}
\begin{proof}
In order to prove the lemma we will show that $s(b)$ as in \eqref{d.8.1} is strictly increasing and satisfies  
\begin{equation}\label{d.8.1.0}
\displaystyle\lim_{b\rightarrow0}s(b)=\dfrac{1}{\Phi(q)}-\dfrac{\delta}{q}\quad \text{and}\quad \lim_{b\rightarrow\infty}s(b)=\infty.
\end{equation}

(i) Since $s'(b)=\dfrac{\delta\Phi(q) e^{-\Phi(q)b}}{q}\mathbb{Z}^{(q)}(b)>0$, for all $b > 0$, we have that $s (\cdot)$ is strictly increasing on $(0,\infty)$. 
	
(ii) Letting $b\rightarrow0$ in \eqref{d.8.1}, it is easy to see that the first limit of \eqref{d.8.1.0} holds.

(iii) Note that 
\begin{equation}\label{d.8.2}
s(b)=\frac{\delta\mathbb{Z}^{(q)}(b)}{q\Phi(q)e^{\Phi(q)b}}\biggr(\dfrac{qe^{\Phi(q)b}\mathbb{Z}^{(\delta\Phi(q))}_{\Phi(q)}(b)}{\delta\mathbb{Z}^{(q)}(b)}-\Phi(q)\biggl).
\end{equation}
Using l'H\^opital's rule, \eqref{W_q_limit}  and that $\varphi(q) > \Phi(q)$, the following limits can be verified:
\begin{align*}
\lim_{b\rightarrow\infty}\dfrac{\mathbb{Z}^{(q)}(b)}{e^{\Phi(q)b}}&=\lim_{b\rightarrow\infty}q[\Phi(q)]^{-1}e^{(\varphi(q)-\Phi(q))b}e^{-\varphi(q)b}\mathbb{W}^{(q)}(b)=\infty,\\
\lim_{b\rightarrow\infty}\dfrac{\mathbb{Z}^{(\delta\Phi(q))}_{\Phi(q)}(b)}{e^{(\varphi(q)-\Phi(q))b}}&=\lim_{b\rightarrow\infty} \delta\Phi(q)\dfrac{e^{-\varphi(q)b}\mathbb{W}^{(q)}(b)}{\varphi(q)-\Phi(q)}=\delta\Phi(q)\dfrac{\psi'_{Y}(\varphi(q))^{-1}}{\varphi(q)-\Phi(q)},\\
\lim_{b\rightarrow\infty}\dfrac{qe^{\Phi(q)b}\mathbb{Z}^{(\delta\Phi(q))}_{\Phi(q)}(b)}{\delta\mathbb{Z}^{(q)}(b)}&=\lim_{b\rightarrow\infty}\dfrac{\Phi(q)\biggr(\dfrac{\mathbb{Z}^{(\delta\Phi(q))}_{\Phi(q)}(b)}{\delta e^{(\varphi(q)-\Phi(q))b}}+e^{-\varphi(q)b}\mathbb{W}^{(q)}(b)\biggl)}{e^{-\varphi(q)b}\mathbb{W}^{(q)}(b)} =\Phi(q)\bigg(1+\dfrac{\Phi(q)}{\varphi(q)-\Phi(q)}\biggr).
\end{align*}
Hence, it follows that $\displaystyle \lim_{b\to\infty}s(b)=\infty$.
\end{proof}

Now, we let $\bar{b}_\Lambda$ be as in Lemma \ref{opt.1} for the case $\Lambda>\dfrac{1}{\Phi(q)}-\dfrac{\delta}{q}$ and set it to zero otherwise.

(i) When $\Lambda>\dfrac{1}{\Phi(q)}-\dfrac{\delta}{q}$, applying \eqref{d.8.0}  in \eqref{d.3}, with $b=\bar{b}_\Lambda$,  we see that $\overline{v}_{\Lambda}^{\bar{b}_\Lambda}$  
is given by
\begin{equation}\label{d.9}
\overline{v}_{\Lambda}^{\bar{b}_\Lambda}(x)=
\begin{cases}
\dfrac{\delta}{q}\mathbb{Z}^{(q)}(\bar{b}_\Lambda-x)-\dfrac{e^{-\Phi(q)(x-\bar{b}_\Lambda)}}{\Phi(q)}\mathbb{Z}^{(\delta\Phi(q))}_{\Phi(q)}(\bar{b}_\Lambda-x),&\text{if}\ x\leq \bar{b}_\Lambda,\\
\dfrac{\delta}{q}-\dfrac{e^{-\Phi(q)(x-\bar{b}_\Lambda)}}{\Phi(q)},&\text{if}\ x>\bar{b}_\Lambda.
\end{cases}
\end{equation}

(ii) When $\Lambda\leq\dfrac{1}{\Phi(q)}-\dfrac{\delta}{q}$,  using \eqref{d.3} and because  $\overline{k}_0(0, \Lambda)=\dfrac{\delta}{q}+\Lambda$, we have
\begin{align}\label{v_0_sp}
\overline{v}^{0}_{\Lambda}(x)=\dfrac{\delta}{q}-e^{-\Phi(q)x}\bigg(\dfrac{\delta}{q}+\Lambda\bigg), \quad x \geq 0.
\end{align}

\begin{theorem}\label{L.V.2}
The optimal strategy for \eqref{p_Lambda_zero} consists of a threshold strategy at level $\bar{b}_\Lambda$.
\end{theorem}
\begin{proof}

In view of \eqref{d.9} and \eqref{v_0_sp}, we confirm that $\overline{v}_{\Lambda}^{\bar{b}_\Lambda}$ is sufficiently smooth. Hence, as  in the spectrally negative case, in order to verify that $\overline{D}^{\bar{b}_{\Lambda}}$ is the optimal strategy over all admissible strategies, it is sufficient to show that the cost function $\overline{v}_{\Lambda}^{\bar{b}_{\Lambda}}$, given by \eqref{d.9} and \eqref{v_0_sp}, satisfies \eqref{equiv_inequality2} and that $\overline{v}_{\Lambda}^{\bar{b}_\Lambda}(0)\geq -\Lambda$.
(i) Suppose $\bar{b}_\Lambda>0$, and so the threshold level $\bar{b}_\Lambda$ satisfies \eqref{d.8.0}. 
From \eqref{d.9} we have that
\begin{equation*}
\overline{v}_{\Lambda}^{\bar{b}_\Lambda\prime}(x)=
\begin{cases}
e^{-\Phi(q)(x-\bar{b}_{\Lambda})}\mathbb{Z}^{(\delta\Phi(q))}_{\Phi(q)}(\bar{b}_\Lambda-x),&\text{if}\ x\leq \bar{b}_\Lambda,\\
e^{-\Phi(q)(x-\bar{b}_\Lambda)},&\text{if}\ x>\bar{b}_\Lambda.
\end{cases}
\end{equation*}
Clearly $\overline{v}^{\bar{b}_\Lambda\prime}_{\Lambda}(x)< 1$ if $x>\bar{b}_\Lambda$. On the other hand, $\overline{v}^{\bar{b}_\Lambda\prime}_{\Lambda}(x)$ is strictly decreasing on $[0,\bar{b}_\Lambda]$ since $ x \mapsto e^{-\Phi(q)(x-\bar{b}_\Lambda)}$ is strictly decreasing  and $x \mapsto\mathbb{Z}^{(\delta\Phi(q))}_{\Phi(q)}(\bar{b}_\Lambda-x)$ is non-increasing in the interval.  This together with $\overline{v}^{\bar{b}_\Lambda\prime}_{\Lambda}(\bar{b}_\Lambda)=1$ shows $\overline{v}^{\bar{b}_\Lambda\prime}_{\Lambda}(x)\geq 1$ if $x\leq \bar{b}_\Lambda$. 

Finally we note that using \eqref{d.9}, \eqref{d.8.1}, and \eqref{d.8.0}
 \[
 \overline{v}^{\bar{b}_\Lambda}_{\Lambda}(0)=\dfrac{\delta}{q}\mathbb{Z}^{(q)}(\bar{b}_\Lambda)-\dfrac{e^{\Phi(q)\bar{b}_\Lambda}}{\Phi(q)}\mathbb{Z}^{(\delta\Phi(q))}_{\Phi(q)}(\bar{b}_\Lambda)=-e^{\Phi(q)\bar{b}_\Lambda} s(\bar{b}_\Lambda)=-\Lambda.
 \]	
	
(ii) Suppose $\overline{b}_{\Lambda}=0$. 
Since $\Lambda\leq\dfrac{1}{\Phi(q)}-\dfrac{\delta}{q}$, it follows that, for $x\geq 0$,
\begin{align*}
	\overline{v}^{0\,\prime}_{\Lambda}(x)&=\Phi(q)e^{-\Phi(q)x}\biggl(\dfrac{\delta}{q}+\Lambda\biggr)\leq e^{-\Phi(q)x}\leq 1.
\end{align*}
Finally, by \eqref{v_0_sp} we obtain that $\overline{v}^{0}_{\Lambda}(0)=-\Lambda$.
\end{proof}

\subsection{Constrained de Finetti's problem for spectrally positive L\'evy processes} \label{section_constrained_dual}

Now we consider the problem \eqref{F.1} driven by the spectrally positive \lev process $\overline{Y}$.
 Note that $\overline{D}^{\bar{b}_\Lambda}$, is the optimal strategy for \eqref{p_Lambda}, for any $K\in[0,1]$.  
 
Let us define $\tilde{\Lambda}:=\frac{1}{\Phi(q)}-\frac{\delta}{q}$. Following Lemma \ref{opt.1}, if $\tilde{\Lambda}\geq0$   we have $\overline{b}_{\Lambda}=0$ for  $\Lambda\in[0,\tilde{\Lambda}]$; on the other hand if $\tilde{\Lambda}<0$ then $\bar{b}_{\Lambda}>0$ for all $\Lambda>0$. 

 Similarly to Section \ref{const.1}, we need to establish the relationship between $\Lambda$ and its corresponding threshold level $\bar{b}_\Lambda$ given by Lemma \ref{opt.1}.  
From \eqref{d.8.0} we get that $\Lambda = \tilde{\lambda}(\bar{b}_\Lambda)$ for $\Lambda > \tilde{\Lambda}$ where 
\begin{equation*}
\tilde{\lambda}(b) =e^{\Phi(q)b} s(b),\ 
\end{equation*}
where $s$ is defined in \eqref{d.8.2}. Since $s$ is strictly increasing (see the proof of Lemma \ref{opt.1}) and satisfies \eqref{d.8.1.0}, it follows immediately that  $ \tilde{\lambda}$ is also strictly increasing,  $\displaystyle\lim_{b\rightarrow\infty} \tilde{\lambda}(b)=\infty$ and 
\begin{equation*}
\lim_{b\rightarrow \overline{b}_{0}} \tilde{\lambda}(b)=
\begin{cases}
\dfrac{1}{\Phi(q)}-\dfrac{\delta}{q},&\text{if}\ \overline{b}_{0}=0,\\
0, &\text{if}\ \overline{b}_{0}>0.
\end{cases}
\end{equation*}
Here, the convergence for the case $\overline{b}_{0} > 0$ holds by 
the fact that $$\displaystyle\lim_{b\to \overline{b}_{0}}\tilde{\lambda}(b)=\lim_{b\to \overline{b}_{0}}e^{\Phi(q)b} s(b)=e^{\Phi(q)\overline{b}_{0}} s(\overline{b}_{0})=0,$$ where the last equality follows because \eqref{d.8.0} and Lemma \ref{opt.1} imply $s(\overline{b}_{0}) = 0$. Note that $\overline{b}_{\tilde{\lambda}(b)}=b$ for all $b>\overline{b}_{0}$.

Next, we need to show that the function $b\mapsto\overline{\Psi}_{x}(b)$, given in \eqref{d.1}, is strictly decreasing with $x>0$ fixed. In case that $x=0$, we see that $\overline{\Psi}_{0}(b)=1$ by \eqref{d.1}. The proof of the following lemma is given in Appendix \ref{proof_mon_psi_SP}.
\begin{lemma}\label{mon_psi_SP}
Let $x>0$ be fixed. Then, the function $b \mapsto \overline{\Psi}_{x}(b)$ defined in \eqref{d.1} is strictly decreasing and satisfies
	\begin{equation*}
	\lim_{b\rightarrow0}\overline{\Psi}_{x}(b)=e^{-\Phi(q)x}\quad\text{and}\quad \lim_{b\rightarrow\infty}\overline{\Psi}_{x}(b)=e^{-\varphi(q)x}.
	\end{equation*}
\end{lemma}

Finally, using similar arguments as in Theorem \ref{main.1} (noting that we have results analogous to Lemmas \ref{Lam.1}(iv) and \ref{decPsi1}), we obtain the following theorem. 
\begin{theorem}\label{main.2}
	Let $x > 0$ be fixed. Then, 
	\begin{equation*}
	\overline{V}(x;K)=
	\begin{cases}
	\overline{v}^{\bar{b}_0}_{0}(x;K),&\text{if}\ K\in\left[\overline{\Psi}_{x}(\bar{b}_0),1\right],\\
	\displaystyle\inf_{\Lambda\geq0} \overline{V}_{\Lambda}(x;K),&\text{if}\ K\in\left(e^{-\varphi(q)x},\overline{\Psi}_{x}(\bar{b}_0)\right),\\
	0,&\text{if}\ K=e^{-\varphi(q)x},\\
	-\infty,&\text{if}\ K\in[0,e^{-\varphi(q)x}).
	\end{cases}
	\end{equation*} 
\end{theorem}

\section{Numerical examples}\label{numerical_section}

In this section, we confirm the obtained results through a sequence of numerical examples for both spectrally negative and positive cases. Throughout this section, we set $q = 0.05$.

\subsection{Spectrally negative case}  We first consider the spectrally negative case as studied in Sections \ref{optimal_dividend_tv_sn} and \ref{const.1}. Here we assume that $X$ is of the form
\begin{equation}
 X_t - X_0= c t+ 0.2 B_t - \sum_{n=1}^{N_t} Z_n, \quad 0\le t <\infty, \label{X_phase_type}
\end{equation}
where $B=\{B_t : t\ge 0\}$ is a standard Brownian motion, $N=\{N_t: t\ge 0\}$ is a Poisson process with arrival rate $\kappa$, and  $Z=\{Z_n; n = 1,2,\ldots\}$ is an i.i.d.\ sequence of exponential variables with parameter $1$ (so that Assumption \ref{assump_completely_monotone} is satisfied). Here, the processes $B$, $N$, and $Z$ are assumed mutually independent.  
We refer the reader to \cite{Egami_Yamazaki_2010_2, KKR} for the forms of the corresponding scale functions. 

We consider the following two parameter sets:
\begin{description}
\item[Case 1] $\kappa = 1$, $\sigma = 0.2$, $c = 1.5$, $\delta = 1$
\item[Case 2] $\kappa = 0.01$, $\sigma = 0$, $c = 5$, $\delta = 0.1$
\end{description}
Here, \textbf{Case 2} corresponds to the case $\bar{\Lambda} = \infty$ where we have $b_\Lambda = 0$ for any choice of $\Lambda$ as in Remark \ref{remark_b_zero_all}.

We first show the optimal solutions for the problem considered in Section \ref{optimal_dividend_tv_sn} focusing on the case $\Lambda = 1$. Figure \ref{figure_xi} plots  the function $b \mapsto \xi_\Lambda(b)$ as in \eqref{xi_lambda} and \eqref{v5}.  Here, in \textbf{Case 1}, it attains a global maximum and the maximizer becomes $b_\Lambda$ by \eqref{b1}. In contrast, in \textbf{Case 2}, it is monotonically decreasing and, by \eqref{b1}, we have $b_\Lambda = 0$.  In Figure \ref{figure_optimality}, we plot the optimal value function $x \mapsto V_\Lambda(x) = v_\Lambda^{b_\Lambda}(x)$ along with the suboptimal value functions $v_\Lambda^{b}$ for the choice of $b = 0,  \bar{b}_\Lambda/2, 3 \bar{b}_\Lambda/2$ for \textbf{Case 1} and $b = 2,4,6$ for \textbf{Case 2}. In both cases, we confirm that $V_\Lambda$ dominates the suboptimal ones uniformly in $x$.

\begin{figure}[htbp]
\begin{center}
\begin{minipage}{1.0\textwidth}
\centering
\begin{tabular}{ccc}
 \includegraphics[scale=0.35]{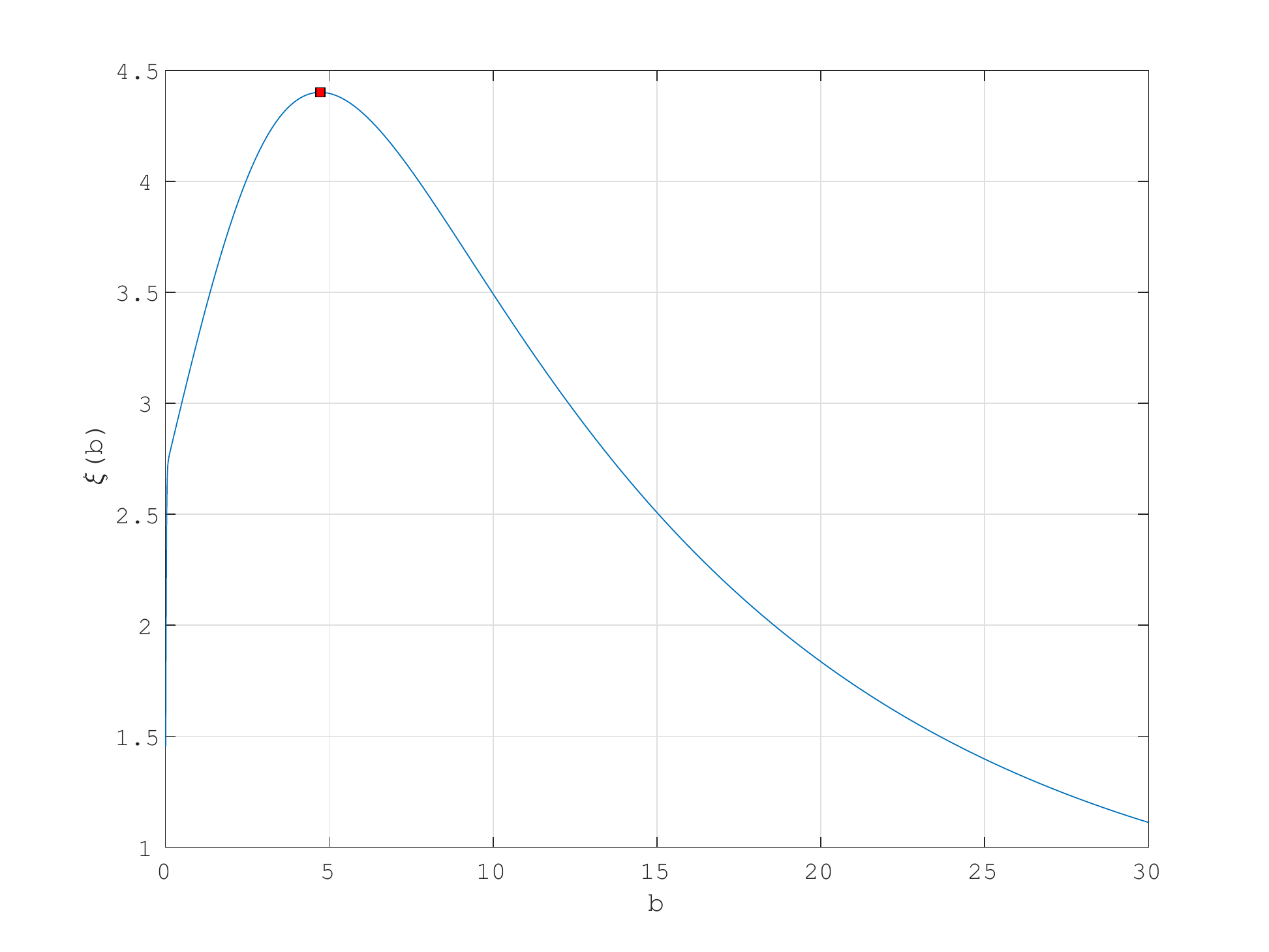} & \includegraphics[scale=0.35]{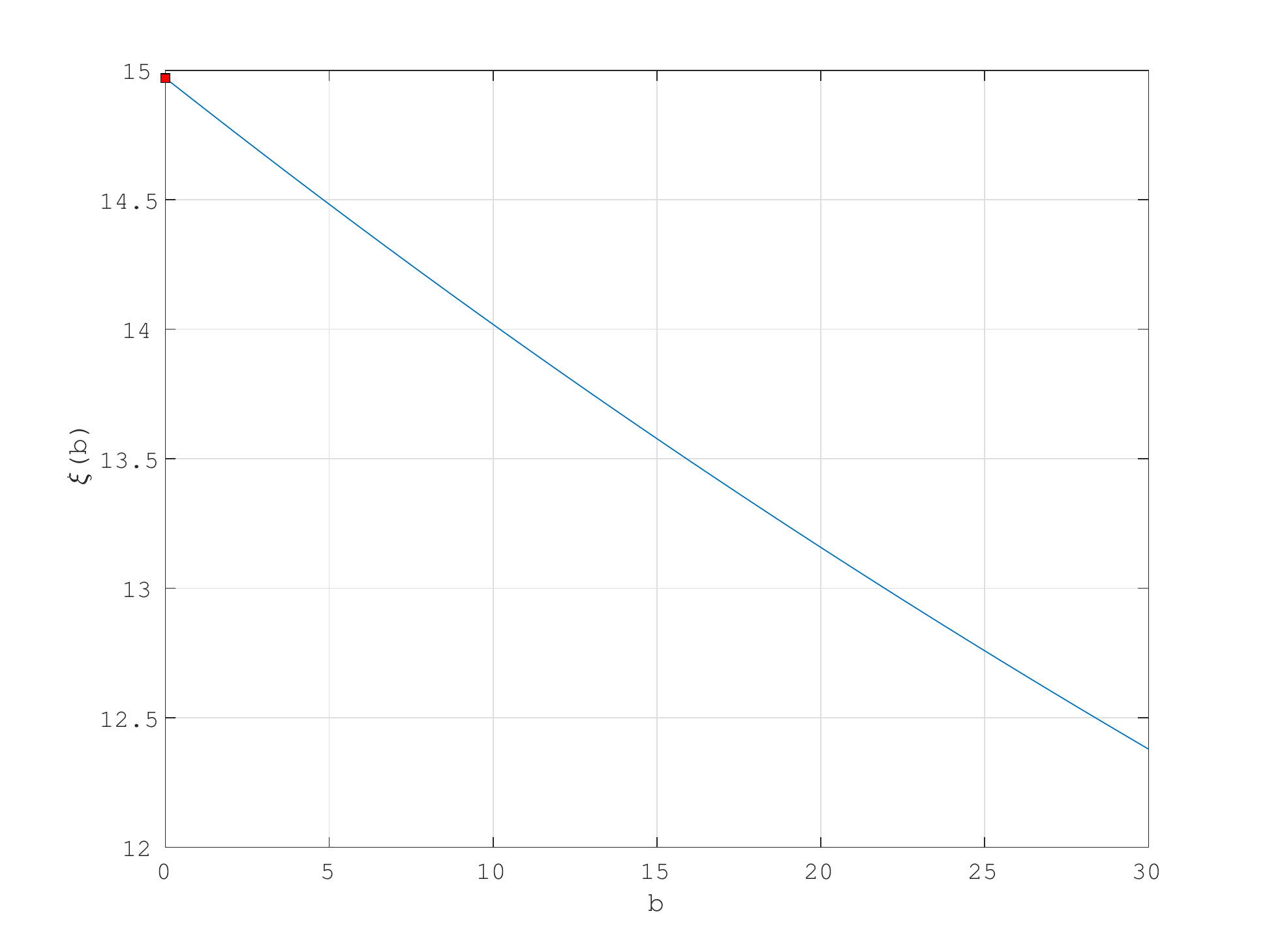}  
 \end{tabular}
\end{minipage}
\caption{Plots $b \mapsto \xi_\Lambda(b)$ for \textbf{Case 1} (left) and \textbf{Case 2} (right). The points at $b_\Lambda$ are indicated by squares.
} \label{figure_xi}
\end{center}
\end{figure}

\begin{figure}[htbp]
\begin{center}
\begin{minipage}{1.0\textwidth}
\centering
\begin{tabular}{ccc}
 \includegraphics[scale=0.35]{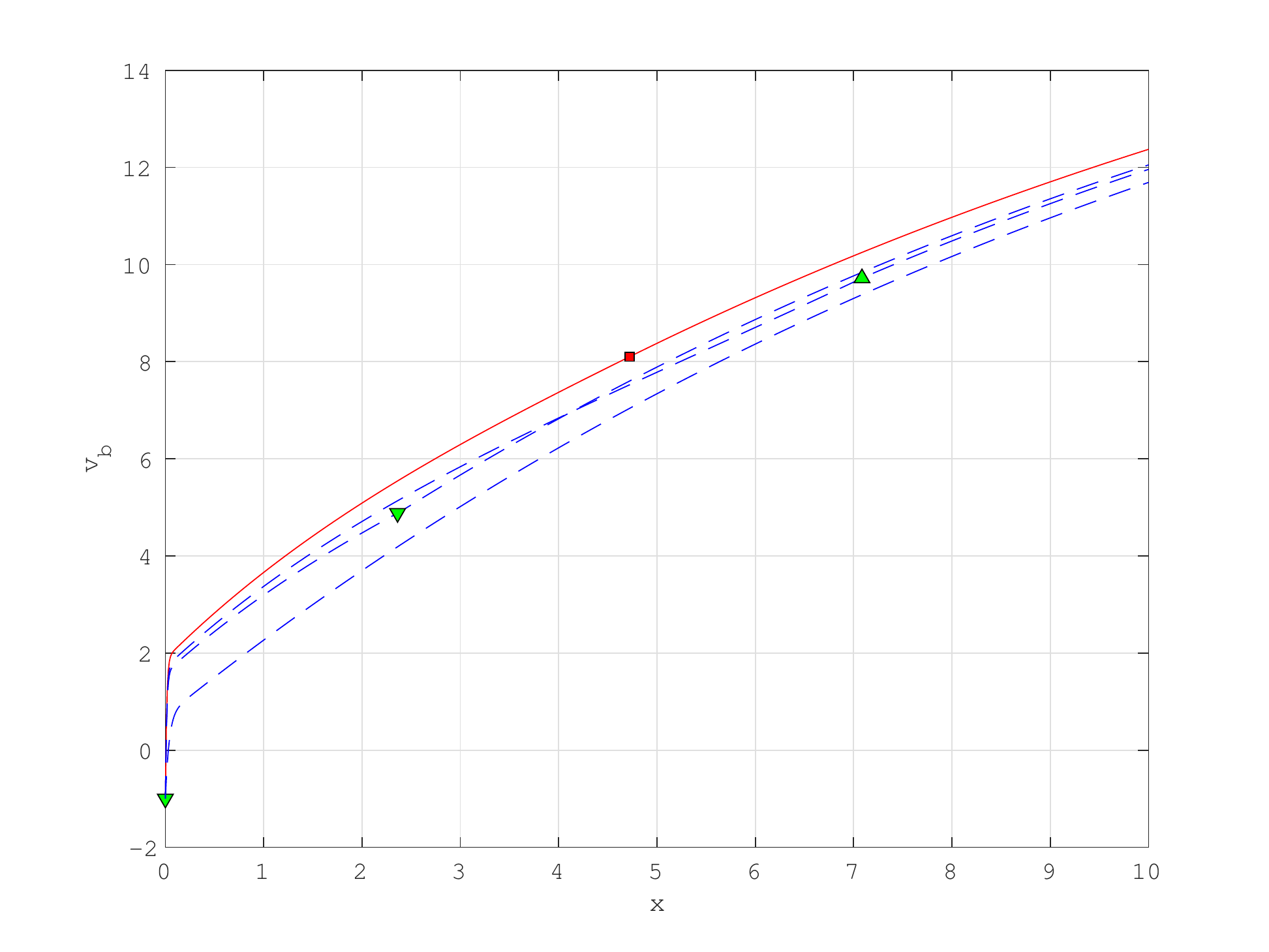} & \includegraphics[scale=0.35]{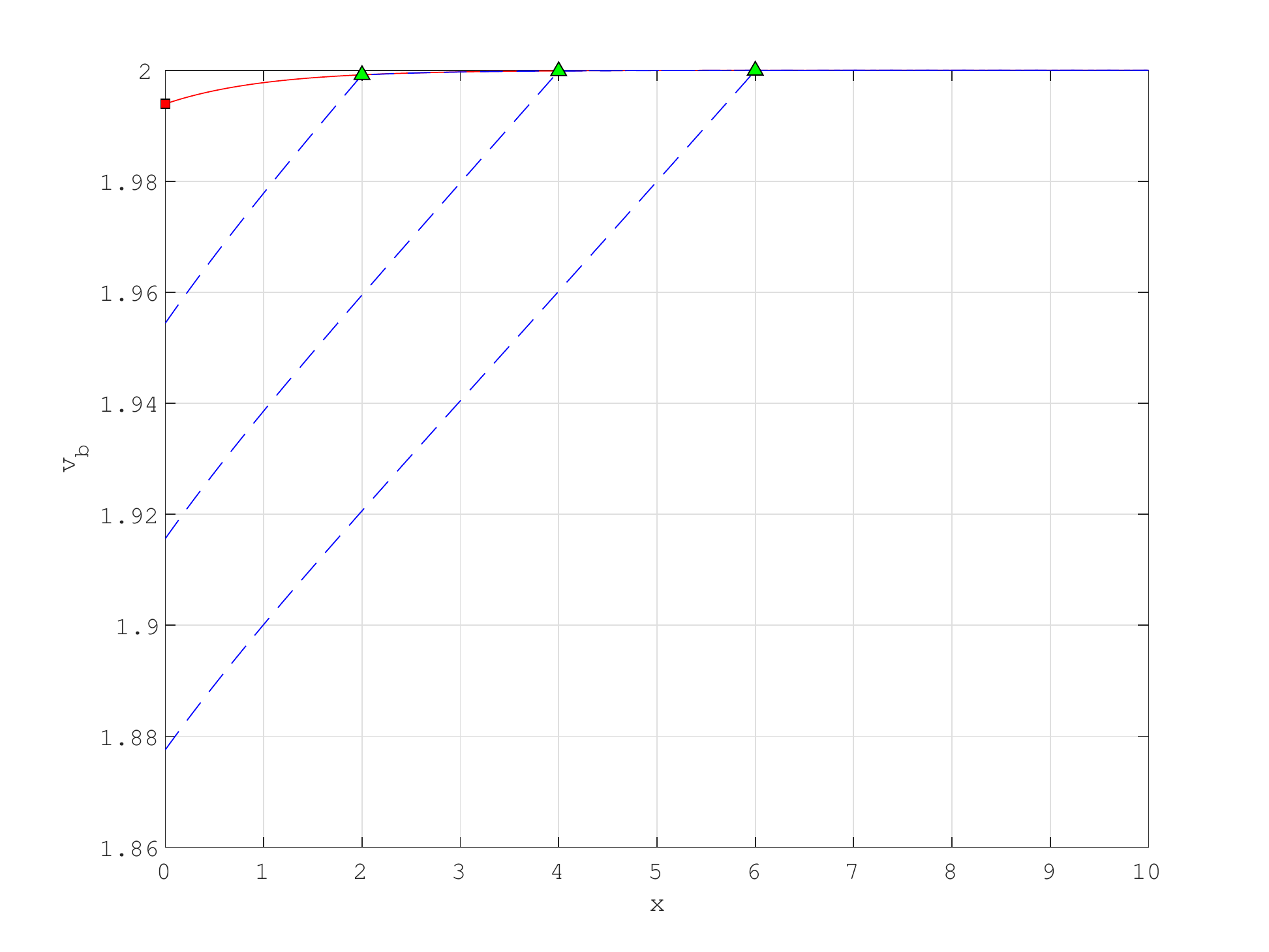}  
 \end{tabular}
\end{minipage}
\caption{Plots  of $x \mapsto V_\Lambda(x)$ (solid)  for \textbf{Case 1} (left) and \textbf{Case 2} (right). Suboptimal value functions $v_\Lambda^{b}$ (dotted) are also plotted 
for the choice of $b = 0,  \bar{b}_\Lambda/2, 3 \bar{b}_\Lambda/2$ for \textbf{Case 1} and $b = 2,4,6$ for \textbf{Case 2}. The points at $b_\Lambda$ are indicated by squares and those at $b$ in the suboptimal cases are indicated by up- (resp.\ down-) pointing triangles when $b > b_\Lambda$ (resp.\ $b < b_\Lambda$).
} \label{figure_optimality}
\end{center}
\end{figure}


We now move onto the constrained problem \eqref{F.1} studied in Section \ref{const.1}, focusing on \textbf{Case 1} with $K = 0.1$. Recall that the optimal solutions are given in Theorem \ref{main.1}. In the left panel of Figure  \ref{figure_Lagrange_2d}, we plot the function $x \mapsto V_\Lambda(x; K) = V_\Lambda(x) + \Lambda K$ for various values of $\Lambda$ ranging from $0$ to $20000$.  For $x \in (\underline{x}, \overline{x})$ where $\underline{x}$ and $\overline{x}$ are such that  $K_{\underline{x}} = K$ and $\Psi_{\overline{x}}(b_0) = K$, respectively, its minimum over the considered $\Lambda$ gives (an approximation of) $V(x; K)$, indicated by the solid red line in the plot.  On the other hand, $V(x;K)$ equals $V_0(x; K) = v_0^{b_0}(x;K)$ for $x \in [\overline{x},\infty)$ and it is infeasible for $x \in [0,\underline{x})$.  On the right panel of Figure  \ref{figure_Lagrange_2d}, we plot, for $x \in (\underline{x}, \overline{x})$, the Lagrange multiplier $\Lambda^* := \displaystyle\arg \min_{\Lambda \geq 0} V_\Lambda(x; K)$. We observe that $\Lambda^*$ goes to infinity as $x \downarrow \underline{x}$ and to zero as $x \uparrow \overline{x}$.

In Figure \ref{figure_Lagrange_3d}, we show the values of $V(x; K)$ and Lagrange multiplier $\Lambda^*$ as functions of $(x, K)$. Here, those $(x, K)$ at which the problem is infeasible are indicated by dark shades on the $z = 0$ plane. It is confirmed that $V(x; K)$ increases as $x$ and $K$ increase, while  $\Lambda^*$ increases as $(x, K)$ decrease.

\begin{figure}[htbp]
\begin{center}
\begin{minipage}{1.0\textwidth}
\centering
\begin{tabular}{ccc}
 \includegraphics[scale=0.35]{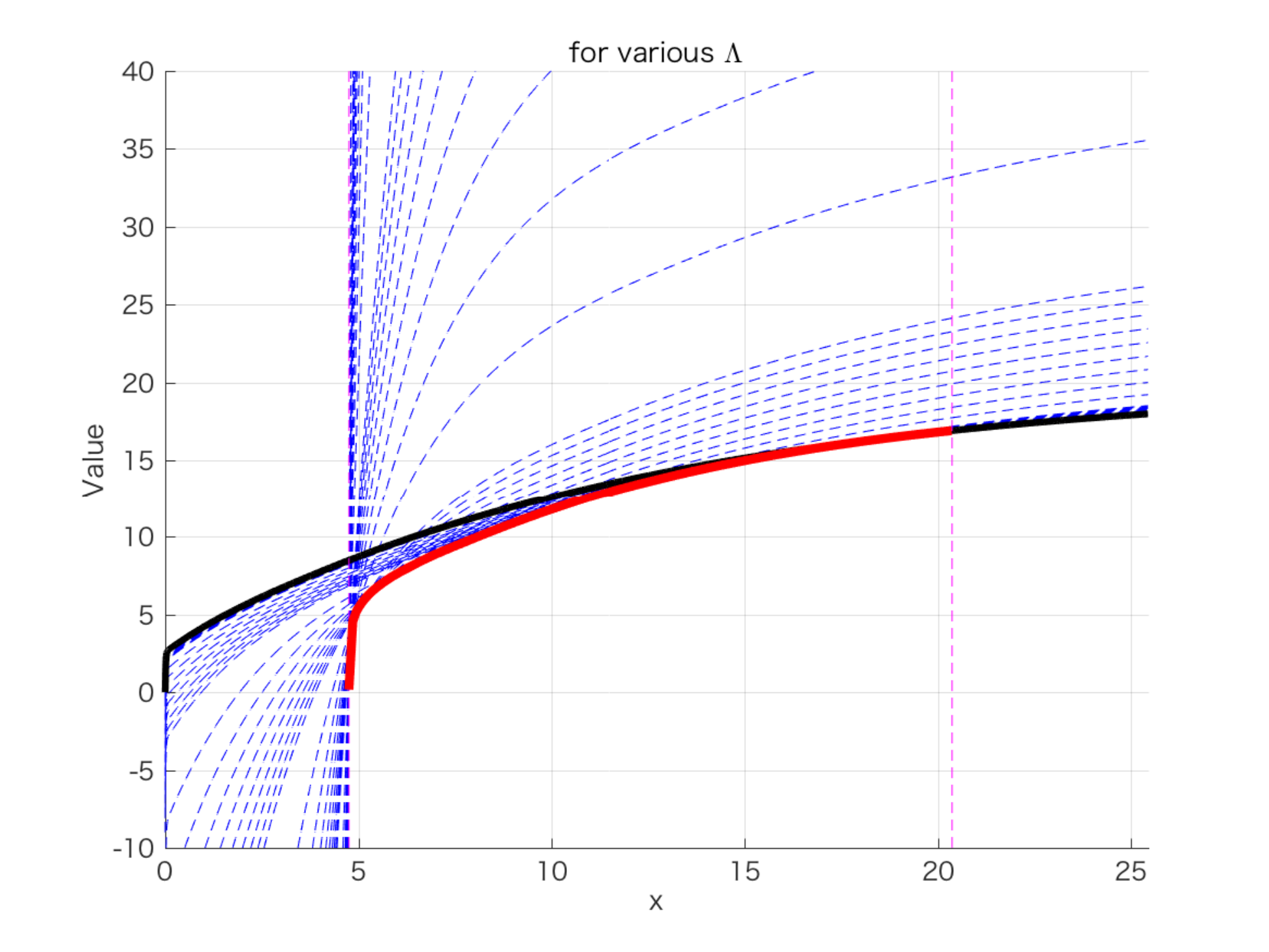} &\includegraphics[scale=0.35]{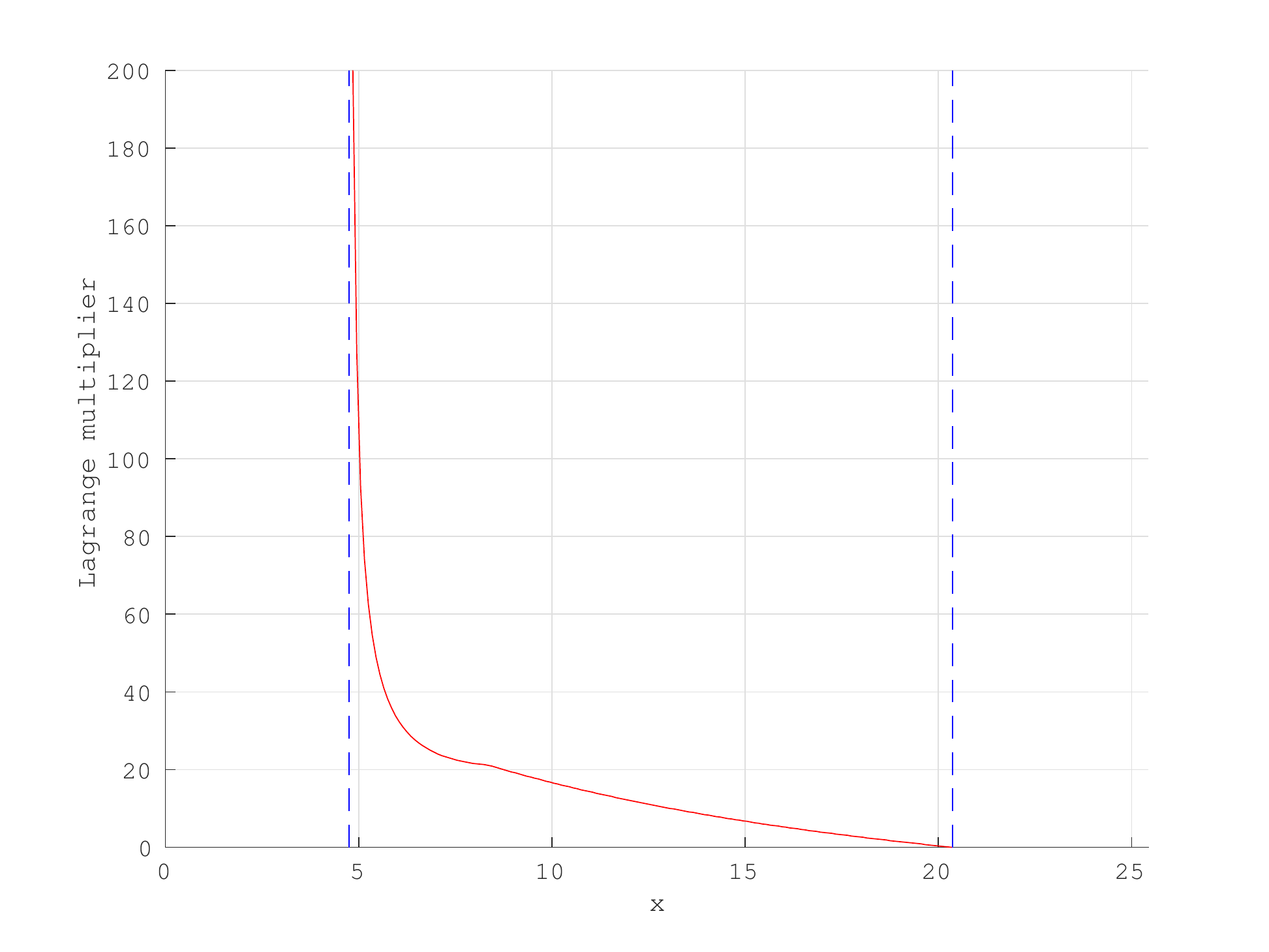}  \\
 \end{tabular}
\end{minipage}
\caption{(Left) Plots of $x \mapsto V_\Lambda(x; K)$ for $\Lambda = 0.1$, $\ldots$, $1$, $2$, $\ldots$, $10$, $20$, $\ldots$, $100$, $200$, $\ldots$, $1000$, $2000$, $\ldots$, $10000$, $20000$ (dotted) and for $\Lambda = 0$ (solid, bold face) for the case $K = 0.1$.  The two vertical dotted lines indicate the values of $\underline{x}$ and $\overline{x}$ such that  $K_{\underline{x}} = K$ and $\Psi_{\overline{x}}(b_0) = K$. On $[\underline{x}, \overline{x}]$, the minimum of $V_\Lambda(x; K)$ over $\Lambda$ is shown in solid fold-face red line. (Right) Plots of the Lagrange multiplier $\Lambda^*$ on $(\underline{x}, \overline{x}]$ with the same two vertical lines as in the left plot.
} \label{figure_Lagrange_2d}
\end{center}
\end{figure}

\begin{figure}[htbp]
\begin{center}
\begin{minipage}{1.0\textwidth}
\centering
\begin{tabular}{ccc}
 \includegraphics[scale=0.35]{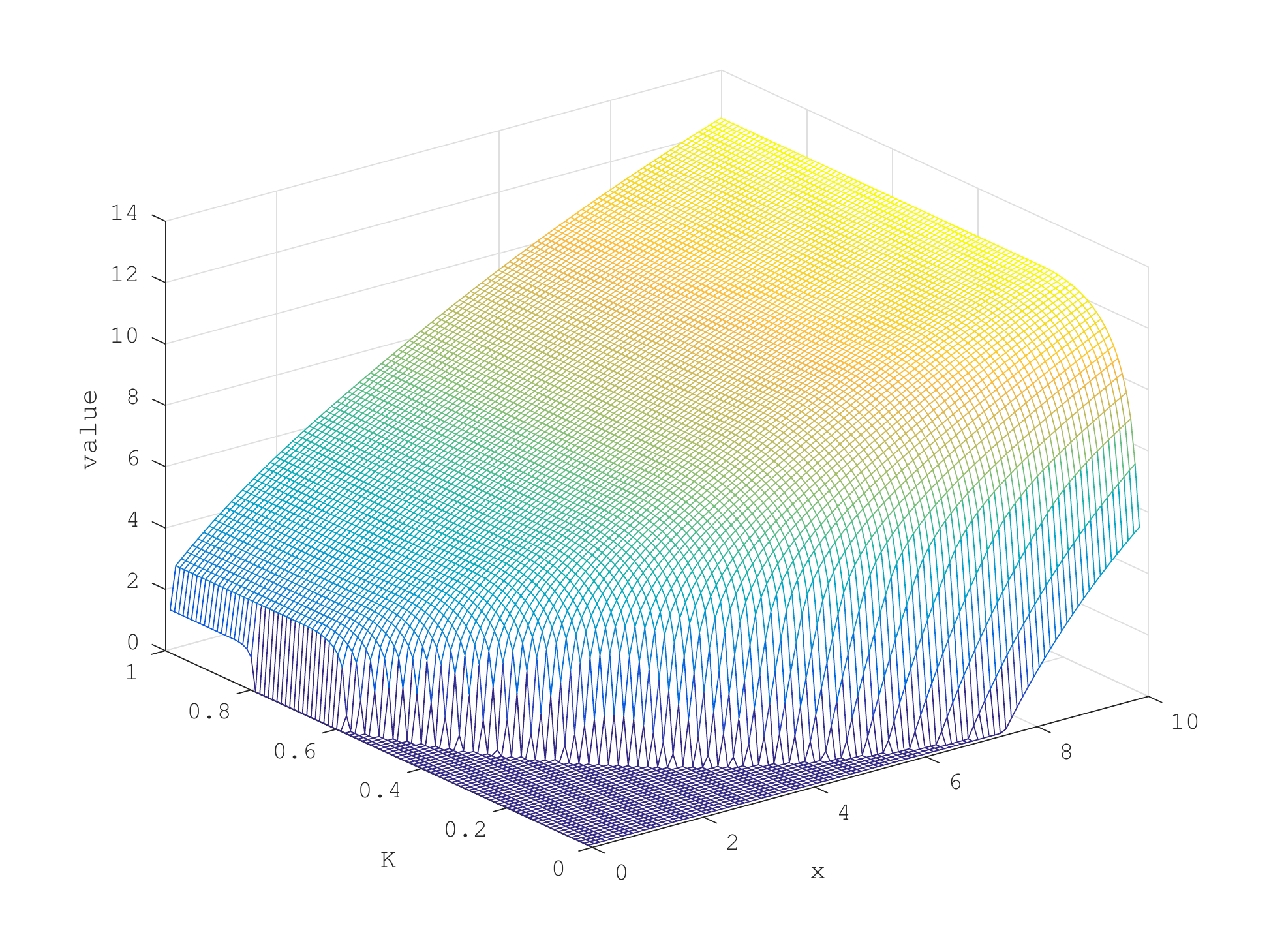} & \includegraphics[scale=0.35]{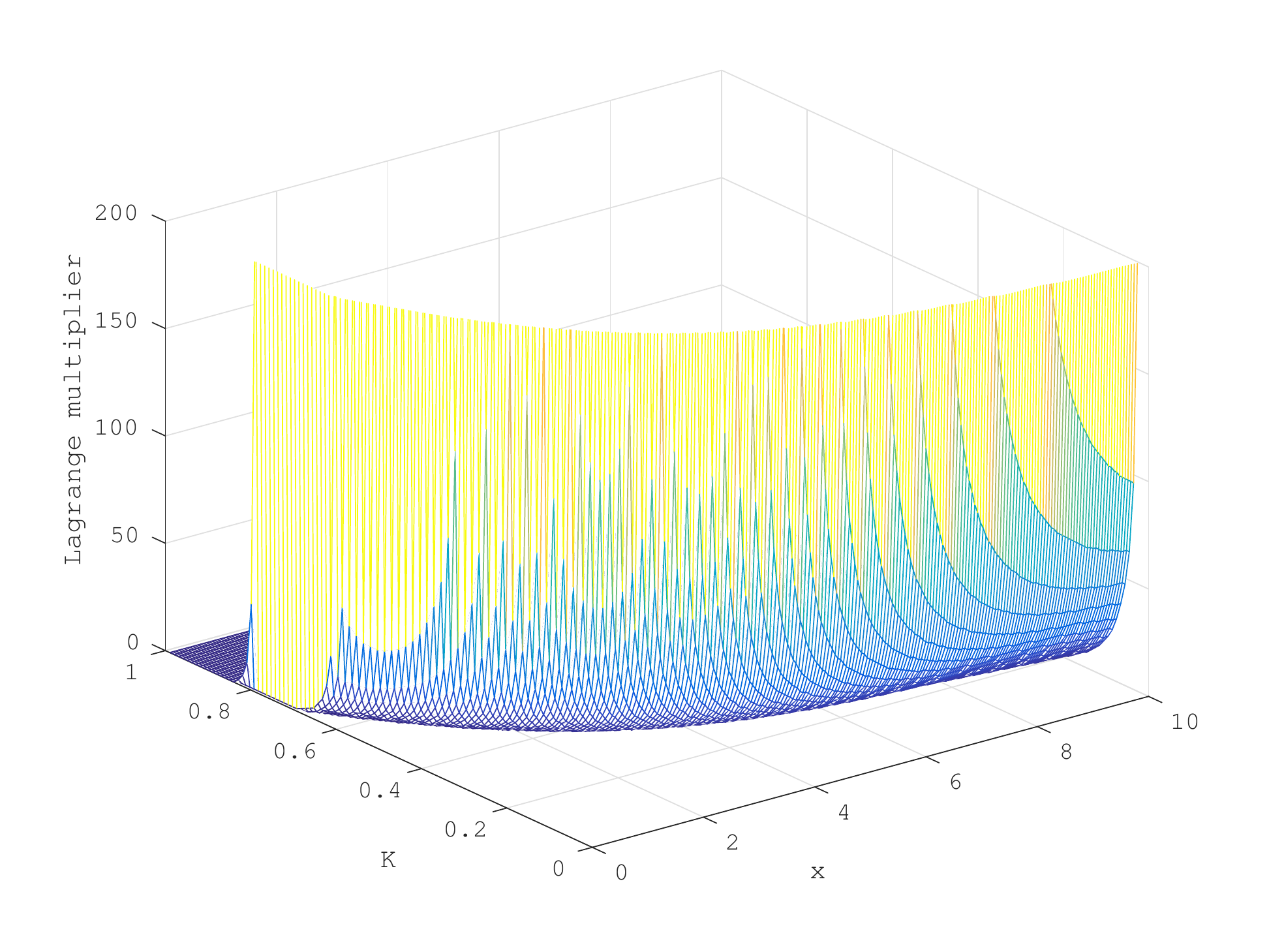}  \\
$(x,K) \mapsto V(x;K)$ & $(x,K) \mapsto \Lambda^*$ \\
 \end{tabular}
\end{minipage}
\caption{Plots of $V(x; K)$ (left) and the Lagrange multiplier $\Lambda^*$ (right) as functions of $x$ and $K$.
} \label{figure_Lagrange_3d}
\end{center}
\end{figure}


\subsection{Spectrally positive case}

Similarly, we confirm the results in Section \ref{specpos} focusing on the case  $\overline{Y}$ is of the form
$$\overline{Y}_t - \overline{Y}_0= - t+ 0.2 B_t + \displaystyle\sum_{n=1}^{N_t} Z_n,\quad \text{for}\ t\geq0. $$ Here $B$ and $N$ (with $\kappa = 1.5$) are the same as in the case of \eqref{X_phase_type}, and $Z$ is a phase-type random variable that approximates the Weibull distribution with shape parameter 2 and scale parameter 1 (see \cite{avanzi2017optimal} for the parameters of the phase-type distribution). Throughout, we set $\delta = 1$.

For the (Lagrangian) problem considered in Section \ref{problem_terminal_value_dual}, the optimal threshold $\bar{b}_\Lambda$ is such that \eqref{d.8.0} holds and the value function $\overline{V}_\Lambda(x) = \bar{v}_\Lambda^{\bar{b}_\Lambda}$ is given in \eqref{d.9}.
Figure \ref{figure_optimality_dual} plots the optimal value function $x \mapsto \overline{V}_\Lambda(x)$ along with suboptimal value functions $\bar{v}_\Lambda^{b}$ for the choice of $b = 0, \bar{b}_\Lambda/2, 3 \bar{b}_\Lambda/2$ when $\Lambda = 1$.  For the constrained case considered in Section \ref{section_constrained_dual},  in Figures \ref{figure_Lagrange_2d_dual} and \ref{figure_Lagrange_3d_dual}, we plot analogous results as  what are shown in Figures \ref{figure_Lagrange_2d} and \ref{figure_Lagrange_3d}, where we assume $K =0.1$ for Figure \ref{figure_Lagrange_2d_dual}. It is confirmed that similar behaviors of the value function and the Lagrange multiplier can be observed as in the spectrally negative case.

\begin{figure}[t]
\begin{center}
\begin{minipage}{1.0\textwidth}
\centering
\begin{tabular}{c}
 \includegraphics[scale=0.35]{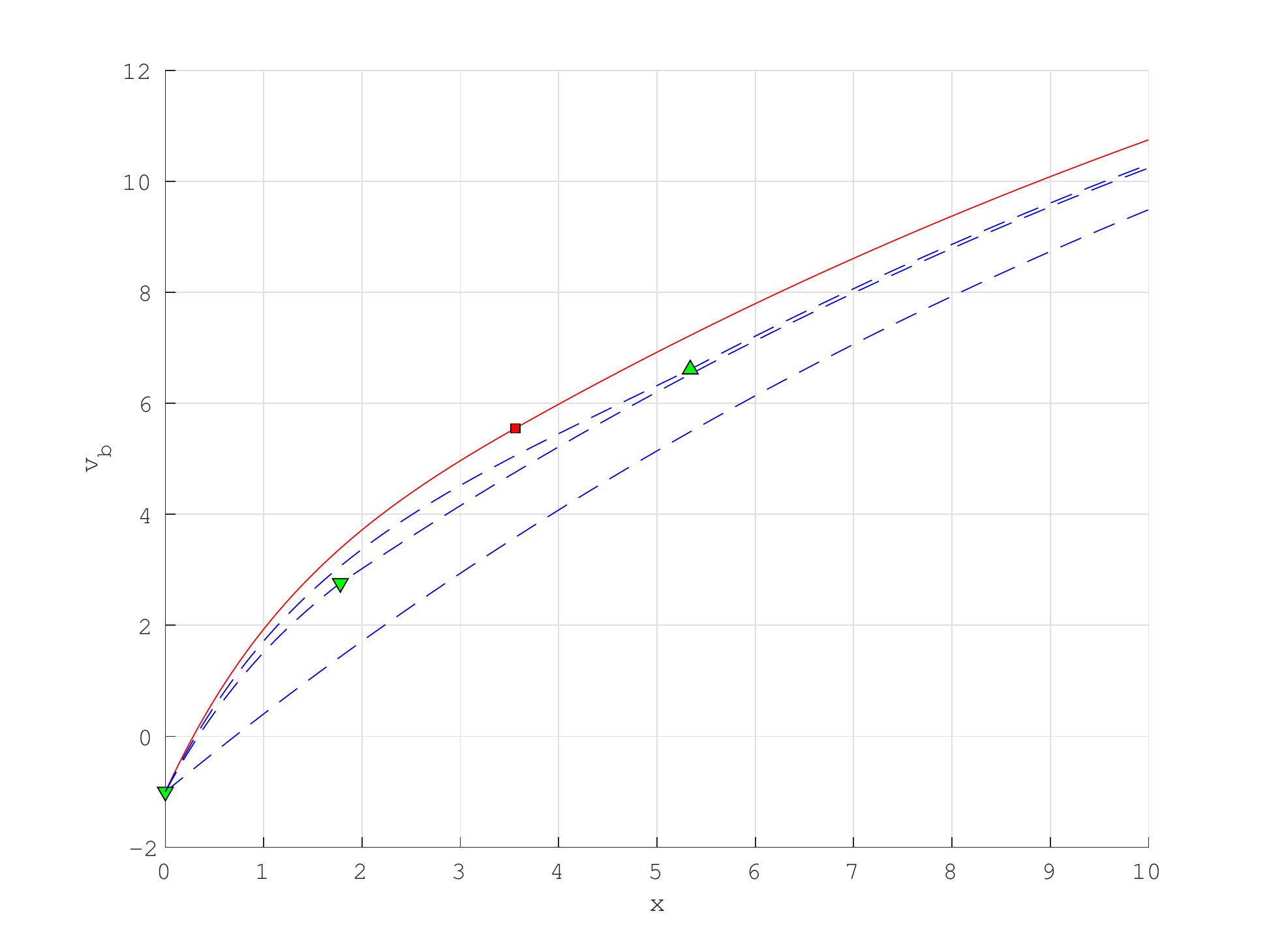} \end{tabular}
\end{minipage}
\caption{Plots  of $x \mapsto \overline{V}_\Lambda(x)$ along with suboptimal value functions $\bar{v}_\Lambda^{b}$ (dotted) 
for the choice of $b = 0,  \bar{b}_\Lambda/2, 3 \bar{b}_\Lambda/2$. The point at $\bar{b}_\Lambda$ is indicated by a square and those at $b$ in the suboptimal cases are indicated by up- (resp.\ down-) pointing triangles when $b > \bar{b}_\Lambda$ (resp.\ $b < \bar{b}_\Lambda$).
} \label{figure_optimality_dual}
\end{center}
\end{figure}

\begin{figure}[htbp]
\begin{center}
\begin{minipage}{1.0\textwidth}
\centering
\begin{tabular}{ccc}
 \includegraphics[scale=0.35]{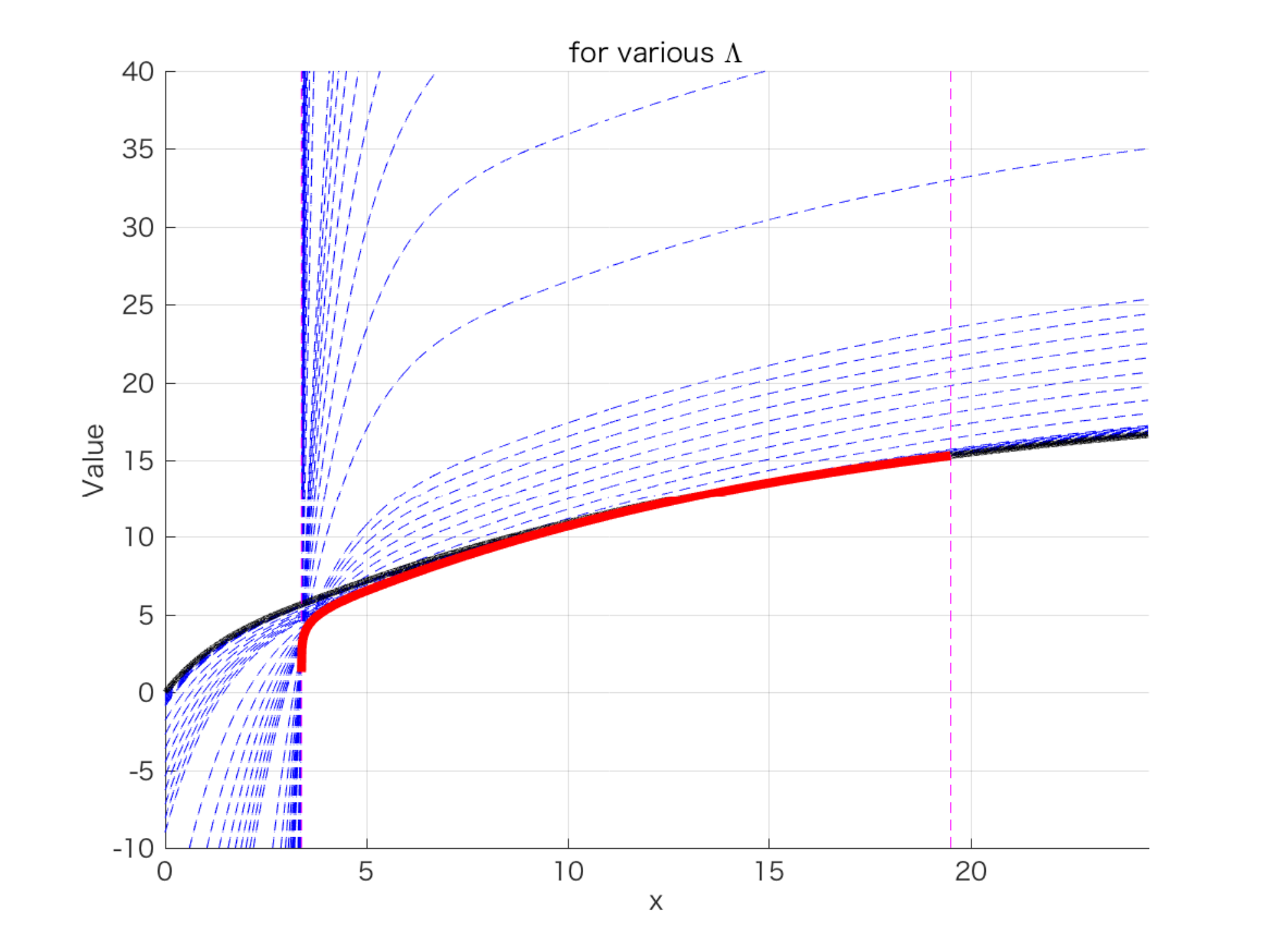} &\includegraphics[scale=0.35]{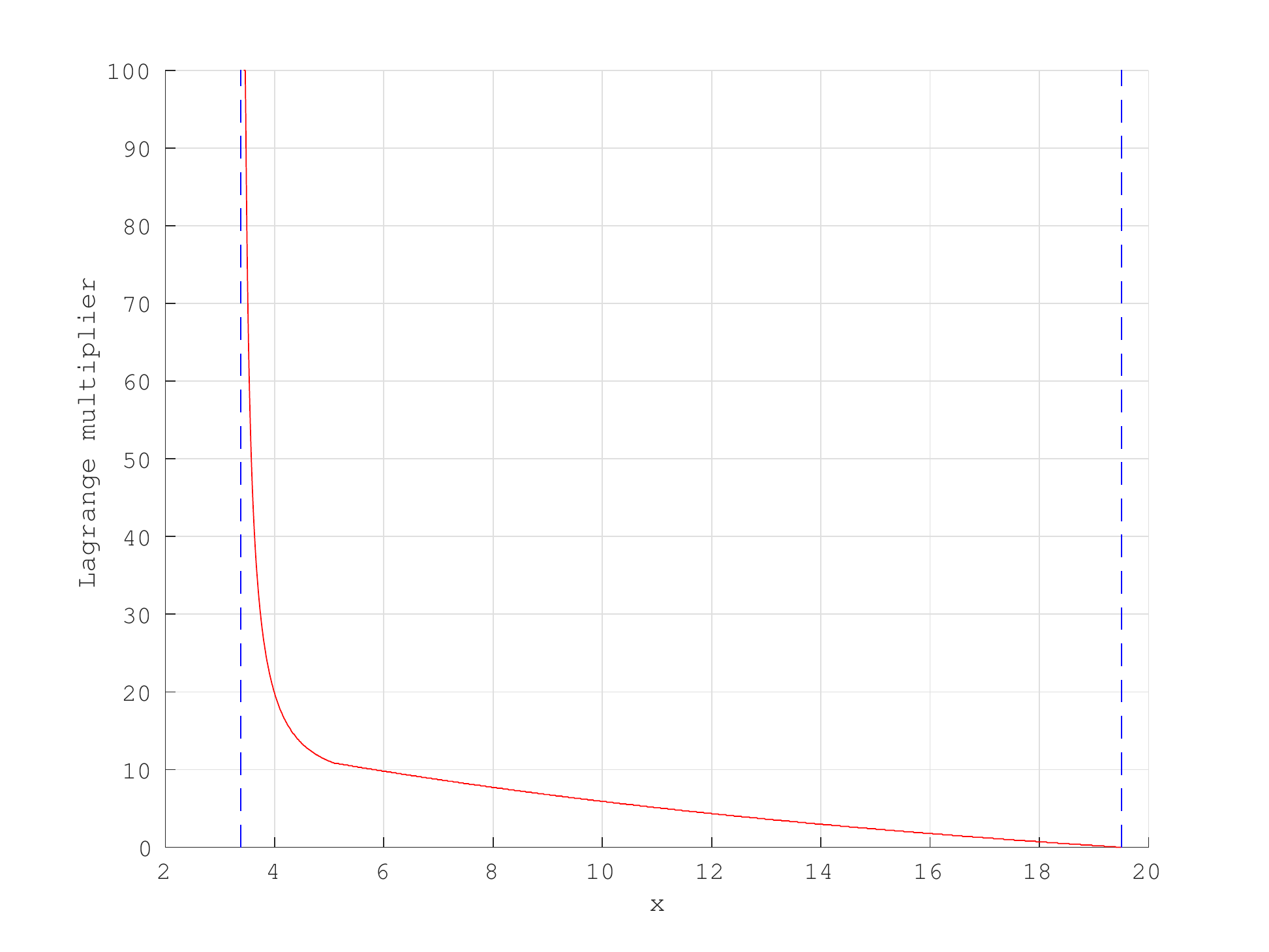}  \\
 \end{tabular}
\end{minipage}
\caption{(Left) Plots of $x \mapsto \overline{V}_\Lambda(x; K)$ for $\Lambda = 0.1$, $\ldots$, $1$, $2$, $\ldots$, $10$, $20$, $\ldots$, $100$, $200$, $\ldots$, $1000$, $2000$, $\ldots$, $10000$, $20000$ (dotted) and for $\Lambda = 0$ (solid, bold face) for the case $K = 0.1$.  The two vertical dotted lines indicate the values of $\underline{x}$ and $\overline{x}$ such that  $\exp (- \varphi(q) \underline{x})= K$ and $\overline{\Psi}_{\overline{x}}(\bar{b}_0) = K$. On $[\underline{x}, \overline{x}]$, the minimum of $\overline{V}_\Lambda(x; K)$ over $\Lambda$ is shown by the solid fold-face red line. (Right) Plots of the Lagrange multiplier $\Lambda^*$ on $(\underline{x}, \overline{x}]$ with the same two vertical lines as in the left plot.
} \label{figure_Lagrange_2d_dual}
\end{center}
\end{figure}

\begin{figure}[htbp]
\begin{center}
\begin{minipage}{1.0\textwidth}
\centering
\begin{tabular}{ccc}
 \includegraphics[scale=0.35]{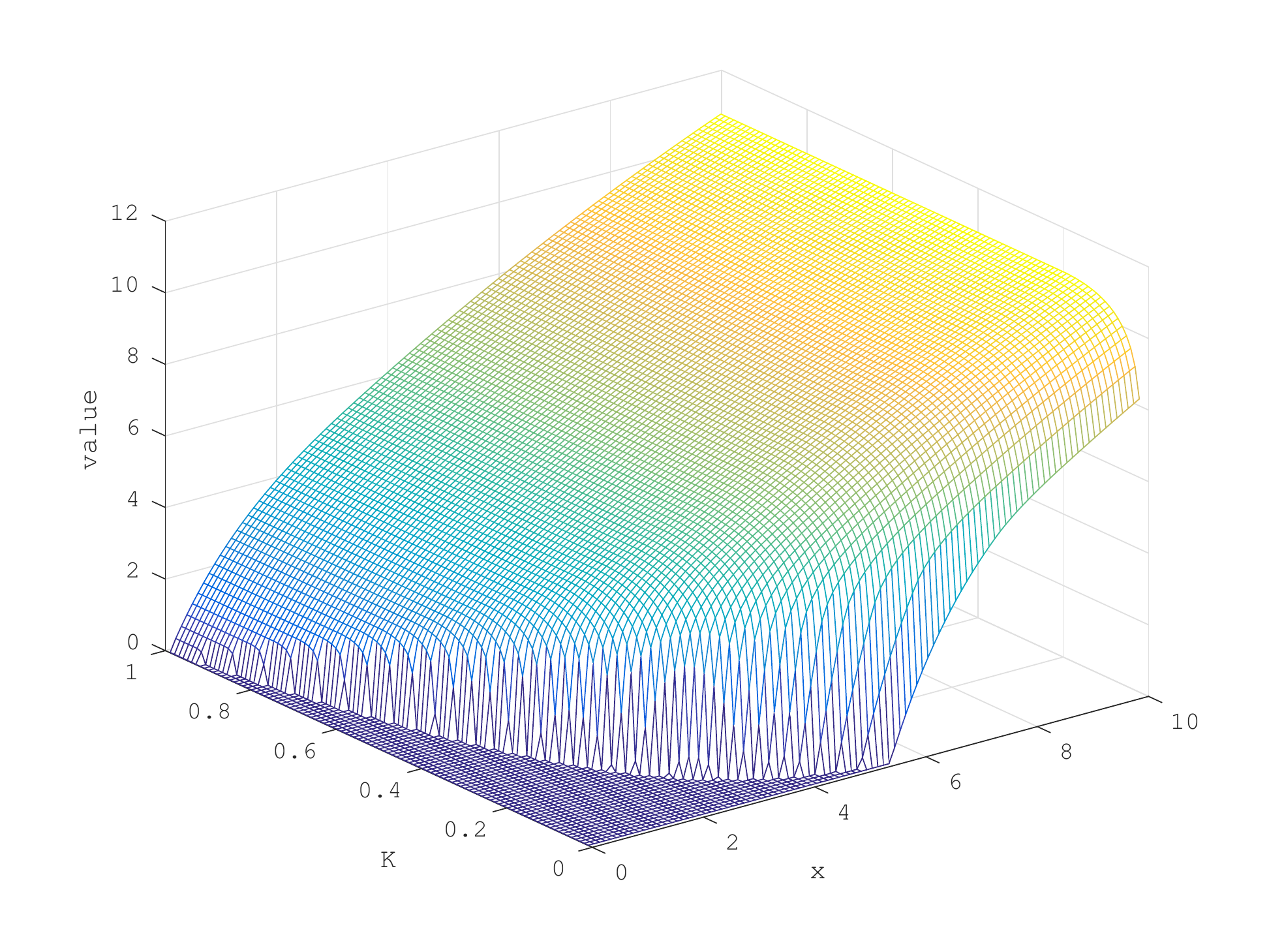} & \includegraphics[scale=0.35]{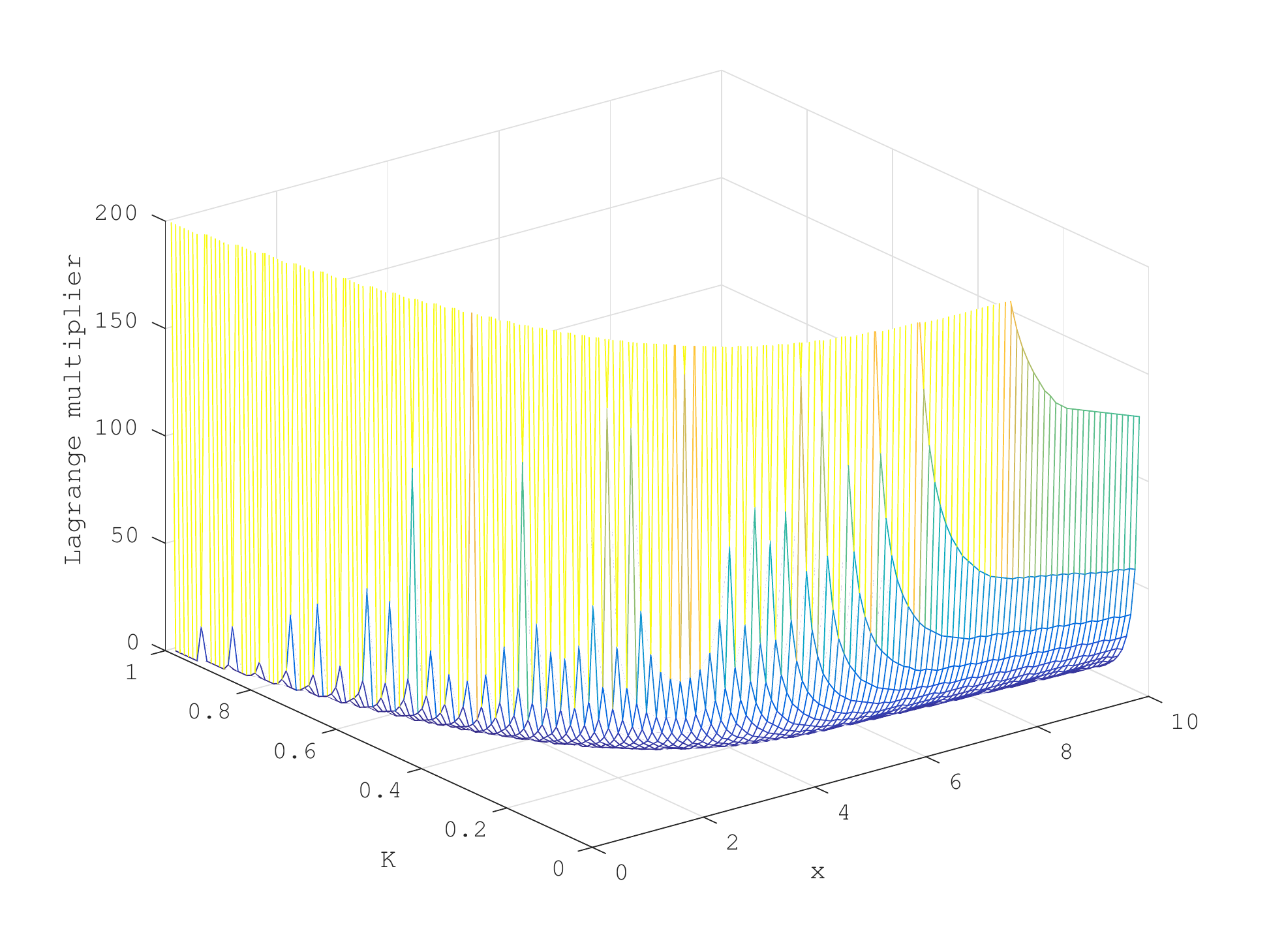}  \\
$(x,K) \mapsto \overline{V}(x;K)$ & $(x,K) \mapsto \Lambda^*$ \\
 \end{tabular}
\end{minipage}
\caption{Plots of $\overline{V}(x; K)$ (left) and the Lagrange multiplier $\Lambda^*$ (right) as functions of $x$ and $K$.
} \label{figure_Lagrange_3d_dual}
\end{center}
\end{figure}

\appendix
\section{Proofs of auxiliary lemmas}

\subsection{Proof Lemma \ref{lemma_zero_case}} \label{proof_lemma_zero_case}
By \eqref{lap_exp_Y}--\eqref{def_varphi}, we have 
\begin{align}
\int_0^\infty e^{-\varphi(q) y} W^{(q)}(y) \diff y = (\delta \varphi(q))^{-1}. \label{laplace_W_varphi}
\end{align}
 Using this and integration by parts,
\begin{align}
\int_0^{\infty} W^{(q) \prime}(y) e^{-\varphi(q) y} \ud y = 
	- W^{(q)}(0) + \varphi(q)\int_0^{\infty} W^{(q)}(y) e^{-\varphi(q) y} \ud y =  - W^{(q)}(0)  + \delta^{-1}. \label{laplace_W_varphi_derivative}
\end{align}
From here, \eqref{xi_0_new} is immediate.

Now, by differentiating \eqref{RLqp} and changing variables,  
\begin{align*}
	\mathbb{W}^{(q)}(x) - W^{(q)}(x) 
	&= \delta \Big( \int_0^x \mathbb{W}^{(q)}(x-y) W^{(q) \prime}(y) \diff y + \mathbb{W}^{(q)}(x) W^{(q)}(0)  \Big).
\end{align*}
This implies
\begin{align} \label{conv_W_W_prime}
	\delta\int_0^x \mathbb{W}^{(q)}(x-y) W^{(q) \prime}(y) \diff y  =  \mathbb{W}^{(q)}(x) - W^{(q)}(x) -  \delta\mathbb{W}^{(q)}(x) W^{(q)}(0). 
\end{align}
Substituting \eqref{RLqp} and \eqref{conv_W_W_prime} in \eqref{vf_1} and after simplification, we obtain \eqref{v_0_new}.

\subsection{Proof of Lemma \ref{com_mon}.} \label{proof_lemma_com_mon}
By integration by parts applied to \eqref{conv_W_W_prime}
\begin{align}
\delta\int_{0}^x\mathbb{W}^{(q)\prime}(x-y)W^{(q)}(y)\diff y=\mathbb{W}^{(q)}(x)-(1+\delta\mathbb{W}^{(q)}(0))W^{(q)}(x). \label{v.2}
\end{align}
On the other hand, by differentiating \eqref{conv_W_W_prime}, we obtain
\begin{align}
\delta\int_{0}^x\mathbb{W}^{(q)\prime}(x-y)W^{(q)\prime}(y)\diff y =(1-\delta W^{(q)}(0))\mathbb{W}^{(q)\prime}(x)-(1+\delta\mathbb{W}^{(q)}(0))W^{(q)\prime}(x).\label{v.3} 
\end{align}		
Applying \eqref{v.2} and \eqref{v.3} in the first equality in \eqref{firstder} it follows that,  for $x > b_\Lambda$,
\begin{align*}
v^{b_{\Lambda}\prime}_{\Lambda}(x)&=g_{\Lambda}(b_{\Lambda})\left((1-\delta W^{(q)}(0))\mathbb{W}^{(q)\prime}(x)-\delta\int_0^{b_{\Lambda}}\mathbb{W}^{(q)\prime}(x-y)W^{(q)\prime}(y)\diff y\right)\notag\\
&\quad-q\Lambda\left(\mathbb{W}^{(q)}(x)-\delta\int_0^{b_{\Lambda}}\mathbb{W}^{(q)\prime}(x-y)W^{(q)}(y)\diff y\right)-\delta\mathbb{W}^{(q)}(x-b_{\Lambda}).
\end{align*}
By \eqref{W_mathbb_completely_monotone}, we can write 
\begin{align*}
v^{b_{\Lambda}\prime}_{\Lambda}(x) = G_1(x) + G_2(x),
\end{align*}
where
\begin{align*}
G_1(x)&:=g_{\Lambda}(b_{\Lambda})\varphi'(q)\varphi(q)e^{\varphi(q)x}\left((1-\delta W^{(q)}(0))-\delta\int_0^{b_{\Lambda}}e^{-\varphi(q)y}W^{(q)\prime}(y)\diff y\right)\\
&\quad-q\Lambda\varphi'(q) e^{\varphi(q)x}\left(1-\delta\varphi(q)\int_0^{b_{\Lambda}}e^{-\varphi(q)y}W^{(q)}(y)\diff y\right)-\varphi'(q)\delta e^{\varphi(q)(x-b_{\Lambda})},\\
G_2(x) &:=-g_{\Lambda}(b_{\Lambda})\left((1-\delta W^{(q)}(0))\hat{f}'(x)-\delta\int_0^{b_{\Lambda}}\hat{f}'(x-y)W^{(q)\prime}(y)\diff y\right)\\
&\quad+q\Lambda\left(\hat{f}(x)-\delta\int_0^{b_{\Lambda}}\hat{f}'(x-y)W^{(q)}(y)\diff y\right)+\delta \hat{f}(x-b_{\Lambda}),
\end{align*}

with $\hat{f}$  as in Remark \ref{com_mon_forW}. Now we note that,  by \eqref{laplace_W_varphi} and \eqref{laplace_W_varphi_derivative},
\begin{align}\label{xi_lambda_2}
\begin{split}
\int_0^{b_{\Lambda}}e^{-\varphi(q)y}W^{(q)\prime}(y)\diff y&=\frac{1}{\delta}-W^{(q)}(0) -\int_{b_{\Lambda}}^{\infty}e^{-\varphi(q)y}W^{(q)\prime}(y)\diff y, \\
\int_0^{b_{\Lambda}}e^{-\varphi(q)y}W^{(q)}(y)\diff y&=\frac{1}{\delta\varphi(q)} -\int_{b_{\Lambda}}^{\infty}e^{-\varphi(q)y}W^{(q)}(y)\diff y.
\end{split}
\end{align}
Combining these, we obtain
\begin{align*}
G_1(x) &=g_{\Lambda}(b_{\Lambda})\varphi'(q)\varphi(q)e^{\varphi(q)x}\delta\int_{b_{\Lambda}}^{\infty}e^{-\varphi(q)y}W^{(q)\prime}(y)\diff y\\
&\quad-q\Lambda\varphi'(q) e^{\varphi(q)x}\delta\varphi(q)\int_{b_{\Lambda}}^{\infty}e^{-\varphi(q)y}W^{(q)}(y)\diff y-\varphi'(q)\delta e^{\varphi(q)(x-b_{\Lambda})}\\
&=\varphi'(q)e^{\varphi(q)(x-b_{\Lambda})}\delta\left(\xi_{\Lambda}(b_{\Lambda})h(b_{\Lambda})-q\Lambda\varphi(q)e^{\varphi(q)b_{\Lambda}}\int_{b_{\Lambda}}^{\infty}e^{-\varphi(q)y}W^{(q)}(y)\diff y\right)\\
&\quad-\varphi'(q)\delta e^{\varphi(q)(x-b_{\Lambda})}=0,
\end{align*}
where in the first equality we used \eqref{xi_lambda_2} and the last one follows from \eqref{xi_lambda}. Hence $v^{b_{\Lambda}\prime}_{\Lambda}(x) = G_2(x)$.
Now, using that $\hat{f}'(x)=- \displaystyle\int_{0+}^{\infty} t e^{-xt}\hat{\mu}(\diff t)$ and Tonelli's theorem, we have \eqref{form_cmf}.
\subsection{Proof of Lemma \ref{decPsi1}.}\label{appendix_decPsc1} 

{(1) We have, by \eqref{h_diff}, 
\begin{align}
\dfrac{\diff}{\diff b} \frac{W^{(q)}(b)}{h(b)} 
= \frac {\alpha(b)} {[h(b)]^2}, \quad b > 0, \label{derivative_frac_W_h}
\end{align}
with 
\begin{align*}
\alpha(b) := W^{(q)\prime}(b) h(b) - W^{(q)}(b) \varphi(q)(h(b)-W^{(q)\prime}(b)) > 0, \quad b > 0,
\end{align*}
where the positivity holds by \eqref{l.5}.

If $x\leq b$, we have that,  by  \eqref{lap_ac.1},
\begin{equation}\label{L.11}
\Psi_{x}(b)=Z^{(q)}(x)-\dfrac{q(W^{(q)}(b)+h(b)/\varphi(q))}{h(b)}W^{(q)}(x). 
\end{equation}
Taking derivative with respect to $b$, by \eqref{derivative_frac_W_h} and the positivity of $\alpha$,
\begin{align*}
	\frac{\diff\Psi_{x}(b)
	}{\diff b} 
	&=-\dfrac{qW^{(q)}(x)}{[h(b)]^{2}}  \alpha(b) < 0.
\end{align*}
Therefore, $\Psi_{x}$ is strictly decreasing on $[x,\infty)$. 

Suppose $b<x$. By \eqref{derivative_frac_W_h}  and \eqref{lap_ac.1},
\begin{align}\label{L.8}
	\frac{\diff\Psi_{x}(b)
	}{\diff b}&=-\delta q \mathbb{W}^{(q)}(x-b)W^{(q)}(b)\notag\\
	&\quad+\frac{q\delta\mathbb{W}^{(q)}(x-b)W^{(q)\prime}(b)(W^{(q)}(b)+h(b)/\varphi(q))}{h(b)}\notag\\
	&\quad-q\left(W^{(q)}(x)+\delta\int_b^x\mathbb{W}^{(q)}(x-y)W^{(q)\prime}(y)\diff y\right)\dfrac{\diff}{\diff b}\biggr[\frac{W^{(q)}(b)+h(b)/\varphi(q)}{h(b)}\biggl]\notag\\
	&=\dfrac{q}{[h(b)]^{2}}
	r(b;x) \alpha(b),
\end{align}
where for $x,b\geq0$ 
\begin{equation*}
	r(b;x):=\dfrac{\delta\mathbb{W}^{(q)}(x-b)h(b)}{\varphi(q)}-W^{(q)}(x)-\delta\int_{b}^{x}\mathbb{W}(x-y)W^{(q)\prime}(y)\diff y. 
\end{equation*}
To prove that $\diff\Psi_{x} / {\diff b}<0$ on $(0,x)$, by the positivity of $\alpha$,  we only need to verify that $r(b;x) < 0$ for all $b\in(0,x)$.

Note that, by \eqref{int_by_parts_laplace},  
\begin{align*}
	\dfrac{\delta\mathbb{W}^{(q)}(x-b)h(b)}{\varphi(q)}=\delta\mathbb{W}^{(q)}(x-b)\biggr(\varphi(q)e^{\varphi(q)b}\int_{b}^{\infty}e^{-\varphi(q)y}W^{(q)}(y)\diff y-W^{(q)}(b)\biggl).
\end{align*}
Using  integration by parts and \eqref{v.2},
\begin{align*}
	-\delta\int_{b}^{x}\mathbb{W}^{(q)}(x-y)W^{(q)\prime}(y)\diff y&=\delta\mathbb{W}^{(q)}(x-b)W^{(q)}(b)-\mathbb{W}^{(q)}(x)\notag\\
	&\quad+W^{(q)}(x)+\delta\int_{0}^{b}\mathbb{W}^{(q)\prime}(x-y)W^{(q)}(y)\diff y.
\end{align*}
Substituting these,
\begin{align*}
	r(b;x)&=-\mathbb{W}^{(q)}(x)+\delta\int_{0}^{b}\mathbb{W}^{(q)\prime}(x-y)W^{(q)}(y)\diff y+\delta\varphi(q)e^{\varphi(q)b}\mathbb{W}^{(q)}(x-b)\int_{b}^{\infty}e^{-\varphi(q)y}W^{(q)}(y)\diff y.
\end{align*}
Now, we rewrite this using Lemma \ref{compmon}. By observing that the terms corresponding to $e^{\varphi(q)x}$ all cancel out (using that $\int_0^\infty e^{-\varphi(q)y} W^{(q)}(y) \diff y = (\delta \varphi(q))^{-1}$), it follows that
\begin{align}\label{L.9}
	r(b;x)&=\hat{f}(x)-\delta\int_{0}^{b}\hat{f}^{'}(x-y)W^{(q)}(y)\diff y-\delta\varphi(q) e^{\varphi(q)b}\hat{f}(x-b)\int_{b}^{\infty}e^{-\varphi(q)y}W^{(q)}(y)\diff y\notag\\
	&=\int_{0}^{\infty}e^{-xt}\biggr(1+\delta t\int_{0}^{b}e^{yt}W^{(q)}(y)\diff y-\delta \varphi(q)e^{b(t+\varphi(q))}\int_{b}^{\infty}e^{-\varphi(q)y}W^{(q)}(y)\diff y\biggl)\hat{\mu}(\diff t).
\end{align}
Taking derivative with respect to $b$ in \eqref{L.9}, it follows that 
\begin{align*}
	\dfrac{\partial r(b;x)}{\partial b}&=\int_{0}^{\infty}e^{-xt}\biggr(\delta (t+\varphi(q))e^{bt}W^{(q)}(b)\notag\\
	&\quad-\delta \varphi(q)(t+\varphi(q))e^{b(t+\varphi(q))}\int_{b}^{\infty}e^{-\varphi(q)y}W^{(q)}(y)\diff y\biggl)\hat{\mu}(\diff t)\notag\\
	&<\int_{0}^{\infty}e^{-xt}\biggr(\delta (t+\varphi(q))e^{bt}W^{(q)}(b)\notag\\
	&\quad-\delta \varphi(q)(t+\varphi(q))e^{b(t+\varphi(q))}W^{(q)}(b)\int_{b}^{\infty}e^{-\varphi(q)y}\diff y\biggl)\hat{\mu}(\diff t)=0,
\end{align*}
where the inequality follows since $W^{(q)}$ is strictly increasing on $(0,\infty)$. From here we have that $r(b;x)$ is strictly decreasing on $(0,x)$. 

On the other hand, by \eqref{L.9} we get that $\displaystyle\lim_{b\rightarrow0}r(b;x)=0$,  and hence  $ r(b;x)<0$ on $(0, x)$. Now we conclude by \eqref{L.8} that $b \mapsto \Psi_{x}(b)$ is also strictly decreasing on $(0,x)$. 

(2) By \eqref{v8}, we see that
\begin{equation}\label{L.13}
\lim_{b\rightarrow\infty}\dfrac{W^{(q)}(b)+h(b)/\varphi(q)}{h(b)}=\dfrac{1}{\Phi(q)}.
\end{equation}
Then, letting $b\rightarrow\infty$ in \eqref{L.11} and using \eqref{L.13}, we obtain the expression in \eqref{L.14}.	
	
\subsection{Proof of Proposition \ref{r.1}.} \label{proof_lemma_r.1}

	Consider $U^{-b}$ the (spectrally negative) refracted \lev process with refraction level $-b$, driven by the process $X$ (as in Section \ref{formProblem}), and $\eta_{-b}:=\inf\{t>0: U^{-b}_{t} > 0 \}$. Then,
\begin{align*}
	\overline{\Psi}_{x}(b)&=\E_{-x}\left[e^{-q\eta_{-b}};\eta_{-b}<\infty\right],\notag\\
	\overline{\E}_{x}\biggr[\displaystyle \int_0^{\overline{\tau}_{b}} e^{-qt}\diff \overline{D}^{b}_{t}\biggl]
	&=\delta\biggr(\dfrac{1}{q}(1-\overline{\Psi}_{x}(b))- \E_{-x}\bigg[ \int_0^{\eta_{-b}} e^{-qt}1_{\{U^{-b}_t  >-b\}}\diff t\bigg]\bigg).
\end{align*}

Now from Theorem 5.(i) of \cite{KyLo}, we have \eqref{d.1}.

On the other hand, by Theorem 6.(iii) in \cite{KyLo}, we obtain
\begin{align*}
	 \E_{-x}\bigg[\int_0^{\eta_{-b}} e^{-qt}1_{\{U^{-b}_t >-b\}}\diff t\bigg]=\Psi_x(b)\overline{\mathbb{W}}^{(q)}(b)-\overline{\mathbb{W}}^{(q)}(b-x).
\end{align*}
Hence, putting the pieces together we get \eqref{d.2}.
\subsection{Proof of Lemma \ref{mon_psi_SP}.} \label{proof_mon_psi_SP}
For $b \in (0, x)$, we have $\overline{\Psi}_{x}(b)=e^{-\Phi(q)x}/{\mathbb{Z}^{(\delta\Phi(q))}_{\Phi(q)}(b)}$, which is clearly strictly decreasing.

To see that it is strictly decreasing on $[x,\infty)$, we need to show that $\dfrac{\diff\overline{\Psi}_{x}(b)}{\diff b}<0$ on $(x,\infty)$. This is satisfied if we can show
	\begin{align}
w(y_{1},y_2) >0,\ y_{1}<y_{2}.\notag
	\end{align}
	where we define, with $y_2$ fixed, 
	\begin{equation*}
w(y,y_2):=\dfrac{\mathbb{Z}^{(\delta\Phi(q))}_{\Phi(q)}(y)}{\mathbb{Z}^{(\delta\Phi(q))}_{\Phi(q)}(y_{2})}-\dfrac{\mathbb{W}^{(\delta\Phi(q))}_{\Phi(q)}(y)}{\mathbb{W}^{(\delta\Phi(q))}_{\Phi(q)}(y_{2})},\ \text{for any}\ y\in\R.
\end{equation*}

Recall the change of measure addressed in Section \ref{section_change_of_measure}. Because,  for $y > 0$, $\{e^{-\delta\Phi(q)(t\wedge\tau_{0,y})}\mathbb{W}^{(\delta\Phi(q))}_{\Phi(q)}(Y_{t\wedge\tau_{0,y}}):t\geq0\}$ and $\{e^{-\delta\Phi(q)(t\wedge\tau_{0,y})}\mathbb{Z}^{(\delta\Phi(q))}_{\Phi(q)}(Y_{t\wedge\tau_{0,y}}):t\geq0\}$ are $\widetilde{\mathbb{P}}^{\Phi(q)}_{x}$-martingales (see Proposition 3 in \cite{Pistorius2004}) where $\tau_{0,y}:=\inf\{t>0:Y_{t}<0\ \text{or}\ Y_{t}>y\}$, it follows that $\{e^{-\delta\Phi(q)(t\wedge\tau_{0,y_{2}})}w(Y_{t\wedge\tau_{0,y_{2}}},y_2) :t\geq0\}$ is a $\widetilde{\mathbb{P}}^{\Phi(q)}_{x}$-martingale. Now, taking $y_{1}<y_{2}$ and using the Optimal Stopping Theorem,
\begin{equation*}
	w(y_{1},y_2)=\widetilde{\E}_{y_{1}}^{\Phi(q)}\Bigr[e^{-q(t\wedge\tau_{0,y_{2}})} w(Y_{t\wedge\tau_{0,y_{2}}},y_2)\Bigl],
\end{equation*}
where $\widetilde{\E}_{y_1}^{\Phi(q)}$ is the expected value with respect to the probability measure $\widetilde{\mathbb{P}}^{\Phi(q)}_{y_1}$. 
Noting that $w$ is bounded (recalling that $Y$ is spectrally negative) taking $t \rightarrow \infty$,  dominated convergence  gives
\begin{equation*}
	w(y_{1},y_2)=\widetilde{\E}_{y_{1}}^{\Phi(q)}\Bigr[e^{-q\tau_{0,y_{2}}} w(Y_{\tau_{0,y_{2}}},y_2) \Bigl].
\end{equation*}
Now we note that
	\begin{itemize}
		\item[(i)] If $Y$ has paths of unbounded variation, then $\mathbb{W}^{(\delta\Phi(q))}_{\Phi(q)}(0)=0$, and hence $w(y,y_2)>0$ for $y\in(-\infty,0]$. On the other hand $\widetilde{\mathbb{P}}^{\Phi(q)}_{x}(Y_{\tau_{0,y_{2}}}\leq 0)>0$.
		\item[(ii)] If $Y$ has paths of bounded variation, then $w(y,y_2)>0$ for $y\in(-\infty,0)$, and $\widetilde{\mathbb{P}}^{\Phi(q)}_{x}(Y_{\tau_{0,y_{2}}}\leq 0)=\widetilde{\mathbb{P}}^{\Phi(q)}_{x}(Y_{\tau_{0,y_{2}}}< 0)>0$.
	\end{itemize}
	These facts imply that
	\begin{equation*}
	w(y_{1},y_2)=\widetilde{\E}_{y_{1}}^{\Phi(q)}\Bigr[e^{-q\tau_{0,y_{2}}}w(Y_{\tau_{0,y_{2}}},y_2) \Bigl]\geq \widetilde{\E}_{y_{1}}^{\Phi(q)}\Bigr[e^{-q\tau_{0,y_{2}}}w(Y_{\tau_{0,y_{2}}},y_2)1_{\{Y_{\tau_{0,y_{2}}}\leq 0\}}\Bigl]>0.
	\end{equation*}
	From here we conclude that $\overline{\Psi}_{x}$ is strictly decreasing on $(0,\infty)$. 

Finally, letting $b\rightarrow0$ in \eqref{d.1}, it is clear that $\displaystyle\lim_{b\rightarrow0}\overline{\Psi}_{x}(b)=e^{-\Phi(q)x}$. On the other hand, using l'H\^opital's rule and \eqref{W_q_limit}, we have that 
	\begin{align*}
		\lim_{b\rightarrow\infty}\frac{\mathbb{Z}^{(\delta\Phi(q))}_{\Phi(q)}(b-x)}{\mathbb{Z}^{(\delta\Phi(q))}_{\Phi(q)}(b)}=\dfrac{1}{e^{(\varphi(q)-\Phi(q))x}}\lim_{b\rightarrow\infty}\dfrac{e^{-\varphi(q)(b-x)}\mathbb{W}^{(q)}(b-x)}{e^{-\varphi(q)b}\mathbb{W}^{(q)}(b)}=\dfrac{1}{e^{(\varphi(q)-\Phi(q))x}}.
	\end{align*}
	Hence $\displaystyle\lim_{b\rightarrow\infty}\overline{\Psi}_{x}(b)=e^{-\varphi(q)x}$. 
\section*{Acknowledgements}

Mauricio Junca was supported by Universidad de los Andes under the Grant Fondo de Apoyo a Profesores Asistentes (FAPA). Harold Moreno-Franco acknowledges financial support from HSE, which was given within the framework of a subsidy granted to the HSE by the Government of the Russian Federation for the implementation of the Global Competitiveness Program. J. L. P\'erez  was  supported  by  CONACYT,  project  no.\ 241195.
K. Yamazaki was supported by MEXT KAKENHI grant no.\  	17K05377.

\end{document}